\documentclass[12pt, reqno]{amsart}
\usepackage{amssymb, mathrsfs}
\usepackage{latexsym}
\usepackage{amsmath}
\usepackage{amsthm}
\usepackage{amsfonts}
\usepackage{dsfont, upgreek}
\usepackage{mathtools}
\usepackage{epsfig}
\usepackage{cite}
\usepackage{amscd}
\usepackage{graphicx}
\usepackage{tikz-cd}
\usepackage{appendix}
\usepackage{enumitem}
\usepackage{color}
\usepackage{todonotes}

\usepackage[left=1.4in,right=1.4in,top=1.2in,bottom=1.2in]{geometry}

\usepackage{tikz}
\usetikzlibrary{decorations.markings}
\usetikzlibrary{arrows.meta}

\usepackage{extarrows}

\usepackage[colorlinks=true,linkcolor=blue,citecolor=blue]{hyperref}

\usepackage[all,cmtip]{xy}
\theoremstyle{plain} % Default
\newtheorem{theorem}{Theorem}[section]
\newtheorem{proposition}[theorem]{Proposition}
\newtheorem{lemma}[theorem]{Lemma}
\newtheorem{corollary}[theorem]{Corollary}

\theoremstyle{definition}
\newtheorem{definition}[theorem]{Definition}

\theoremstyle{remark}
\newtheorem{remark}[theorem]{Remark}

\numberwithin{equation}{section}

\newcommand{\dom}[1]{\hat{#1}}
\newcommand{\norm}[1]{\lVert #1 \rVert}
\newcommand{\onorm}[1]{\Vert #1 \Vert}
\newcommand{\pnorm}[1]{| #1 |}
\newcommand{\bnd}[1]{\partial #1}
\newcommand{\bund}[1]{\mathbb{#1}}
\newcommand{\cat}[1]{\mathcal{#1}}
\newcommand{\cbar}[1]{\dom{c}\left(#1\right)}
\newcommand{\field}[1]{\mathbb{#1}}
\newcommand{\abs}[1]{\lvert #1 \rvert}
\newcommand{\R}{\field{R}}
\newcommand{\C}{\field{C}}

\newcommand{\sphe}{\mathbb{S}}
\newcommand{\Sc}{\mathrm{Sc}}
\newcommand{\Ric}{\mathrm{Ric}}

\newcommand{\inner}[1]{\langle #1 \rangle}
\newcommand{\binner}[1]{\big< #1 \big>}

\NewDocumentCommand{\cone}{o m}{C\IfValueT{#1}{_{#1}}#2}
\NewDocumentCommand{\curvtens}{o}{R\IfValueT{#1}{^{#1}}}
\NewDocumentCommand{\curvoper}{o}{\mathcal{R}\IfValueT{#1}{^{#1}}}
\NewDocumentCommand{\conn}{o}{\nabla\IfValueT{#1}{^{#1}}}
\NewDocumentCommand{\dvol}{o}{\,dvol\IfValueT{#1}{_{#1}}}

\newcommand{\Cl}{C\ell}

\newcommand{\hdim}{h} % Horizontal dimension

\newcommand{\ldim}{l} % Link dimension
\newcommand{\tdim}{n} % Total dimension

\DeclareMathOperator{\trac}{tr} % Operator trace
\DeclareMathOperator{\dist}{dist}

\DeclareMathOperator{\operdom}{dom}
\DeclareMathOperator{\dive}{div}

\begin{document}
\title[scalar-mean comparison for manifolds with conical strati]{A scalar-mean curvature comparison theorem for manifolds with iterated conical singularities}

\author{Milan Jovanovic}
\address[Milan Jovanovic]{Texas A\&M University}
\email{milankj@tamu.edu}
\author{Jinmin Wang}
\address[Jinmin Wang]{Institute of Mathematics, Chinese Academy of Sciences}
\email{jinmin@amss.ac.cn}

\thanks{}

\date{}
\begin{abstract}
We use the Dirac operator method to prove a scalar-mean curvature comparison theorem for spin manifolds which carry iterated conical singularities. Our approach is to study the index theory of a twisted Dirac operator on such singular manifolds.
A dichotomy argument is used to prove the comparison theorem without knowing precisely the index of the twisted Dirac operator.
 This framework also enables us to prove a rigidity theorem of Euclidean domains and a spin positive mass theorem for asymptotically flat manifolds with iterated conical singularities.
\end{abstract}
\maketitle
\section{Introduction}
The study of comparison theorems in differential geometry, particularly those involving scalar curvature and mean curvature, has seen significant developments in recent decades.
In \cite{Gromov:2021}, Gromov proposed studying comparison theorems of scalar curvature alongside its companion, mean curvature. Lott, in \cite{Lott:2020}, established a scalar-mean comparison theorem for compact spin manifolds with boundary via Dirac operators, showing that for a model manifold with nonnegative curvature operator, nonnegative second fundamental form, and nonzero Euler characteristic, it is impossible to simultaneously increase its scalar curvature, mean curvature, and metric tensor. This result generalizes the scalar curvature comparison theorems of Llarull \cite{Llarull:1998} and Goette--Semmelmann \cite{Goette:2002} for closed manifolds. This Dirac operator method was also employed by Cecchini--Zeidler \cite{CecchiniZeidler2024} to show a scalar--mean comparison theorem for Riemannian bands. Meanwhile, using the method of \(\mu\)-bubbles, Wang, Zhu, and the second author \cite{WZ24} derived a Listing-type scalar-mean comparison result in low dimensions without assuming the spin condition.

The investigation of scalar-mean comparison theorems naturally extends to singular manifolds.
Gromov in \cite{Gromov:2021} posed the question whether Llarull's theorem still holds for spheres with subsets removed. For the case of antipodal points, this was answered in the affirmative, proven independently by \cite{MR4733718} and \cite{WangXie25}. If higher-dimensional subsets of the sphere are removed, then further regularity conditions are required \cite{cecchini2022lipschitz,WX24linf,CLZ24,CWXZ,Xiequant}.
Meanwhile, comparison theorems for polyhedra have been explored \cite{Gromov:2021,Wang:2023,MR4689374,WXwarped}. Notably, He--Shi--Yu in \cite{HeShiYu25} establish a minimal-surface approach for singular spaces.

In this paper, we are interested in a generalization of the comparison theorems where the model manifold itself carries \emph{non-isolated} asymptotically conical singularities.
This is motivated by Cheeger's foundational work \cite{MR730920}. Briefly speaking, a manifold \((M,g)\) carries asymptotically conical singularities if its metric near singularities locally takes the form
\begin{equation*}
    g=g_B+dr^2+r^2g_F(b)+\textup{higher order terms}
\end{equation*}
where \(g_B\) is a smooth metric pulled back from the base, \(r\) is the defining function that measures the distance from the singularity, and \(g_F(b)\) is a metric on the link fiber varying with the base, which may itself contain lower-dimensional conical singularities. The reader may refer to Sections \ref{sec:manifolds with facs} and \ref{sec:manifolds with iacs} for the precise definition. For manifolds with boundary, we assume that the singularities lie in the interior. Our main result is as follows.
\begin{theorem}\label{thm:intro-scalar curvature comparison}
	Let \((\dom{M}^{\tdim},\bnd{\dom{M}}^{\tdim-1},\dom{g})\) and \((M^{\tdim},\bnd{M}^{\tdim-1},g)\) be compact spin manifolds with iterated asymptotically conical singularities (IACS) and smooth boundaries \(\bnd{\dom{M}}\), \(\bnd{M}\).
    Assume that the singular points occur in codimension at least three.
	Suppose that the curvature operator of \(M\) and the second fundamental form on the boundary of \(M\) are nonnegative.
	Let \(F \colon \dom{M} \to M\) be an asymptotically canonical map, with \(F_\partial\) denoting its restriction to the boundary.
    Assume that one of the following conditions holds:
	\begin{enumerate}[label=\((\roman*)\)]
		\item The boundaries of \(\dom{M}\) and \(M\) coincide, namely, \(\bnd{\dom{M}} = \bnd{M}\) and \(g|_{\bnd{\dom{M}}}\) equals \(g|_{\bnd{M}}\)
		\item The Euler characteristic of \(\bnd{M}\) is nonzero and \(F\) has nonzero degree.
	\end{enumerate}
	If the following comparison conditions hold:
	\begin{enumerate}[label=\textnormal{(\arabic*)}]
		\item \(\Sc_{\dom{g}} \geq \onorm{dF}^2 F^*\Sc_g\)
		\item \(H_{\dom{g}} \geq \onorm{d(F_\partial)} (F_\partial)^*H_g\)
	\end{enumerate}
	then, we conclude
	\begin{enumerate}[label=\textup{(\Roman*)}]
		\item \(\Sc_{\dom{g}} = \onorm{dF}^2 F^* \Sc_g\)
		\item \(H_{\dom{g}} = \onorm{d(F_\partial)} (F_\partial)^*H_g\)
	\end{enumerate}
	Moreover,
	\begin{enumerate}[label=\((\mathrm{\Roman*}^\prime)\)]
		\item If \(\Ric_g > 0\), then \(\onorm{dF} \equiv a\) for some constant \(a > 0\) and \(F \colon (\dom{M}, a \cdot \dom{g}) \to (M,g)\) is a Riemannian covering map.
		\item If \((M, g)\) is flat, then \((\dom{M}, \dom{g})\) is Ricci flat.
	\end{enumerate}
\end{theorem}

To prove Theorem \ref{thm:intro-scalar curvature comparison}, we analyze a twisted Dirac operator \(D\) on \(\dom{M}\) with a specific boundary condition.
The study of \(D\) presents two major challenges.
\begin{enumerate}[label=(\alph*)]
	\item \emph{Fredholmness}: showing that \(D\) is Fredholm under the given boundary condition;
	\item \emph{Index computation}: determining the Fredholm index of \(D\).
\end{enumerate}

For spin Dirac operators or signature operators on manifolds with conical singularities, the essential self-adjointness and Fredholmness follow from a Witt-type condition \cite{Chou:1985,Bruning:1988,MR1449639,PierreJesse16,PiazzaPauloRosenberg23}, which requires that the associated link operators admit a spectral gap at least \(1/2\).
We demonstrate that this criterion extends to the twisted Dirac operator, assuming the scalar curvature comparison and the nonnegativity of the curvature operator on the target manifold.
Strictly speaking, one must assume that the conical singularities occur in codimension at least three. Indeed, the theorem fails for manifolds with codimension-two singularities (see Remark \ref{rem:counterexample to comparison theorem when 1-dimensional links are present}).

For manifolds with conical singularities, the index formula of Dirac operators typically involves the eta invariants arising from the singularities \cite[Main Theorem]{PierreJesse16}, which are by nature spectral and nonlocal.
In this work, we give a dichotomy argument to establish the comparison theorems \emph{without} having to know the precise index of the twisted Dirac operator. This approach is inspired by Witten's proof of the spin positive mass theorem \cite{MR626707}. As an application, we establish a scalar-mean rigidity theorem for Euclidean domains.

\begin{theorem}\label{thm:rigidity of NNSC fill-in - intro}
    Let \((M^{\tdim},\bnd{M}^{\tdim-1},g)\) be a compact spin manifold with IACS.
    Let \(F_\partial\colon\partial M\to\Sigma\) be a smooth map with nonzero degree, where \(\Sigma\) is a closed convex hypersurface in the Euclidean space $\R^n$.
    Assume that one of the following conditions holds:
    \begin{enumerate}[label=\((\roman*)\)]
        \item \(F_\partial\) is an isometry.
        \item \(n\) is odd.
    \end{enumerate}
    Let $H$ be the mean curvature of $\Sigma$. If \(\Sc_{g} \geq 0\) and \(H_{g} \geq \onorm{dF_\partial} (F_\partial)^*H\), then \(g\) is flat.
\end{theorem}

Moreover, we extend this idea to prove a positive mass theorem for asymptotically flat (AF for short) spin manifolds with iterated asymptotically conical singularities. The famous positive mass theorem states that an asymptotically flat smooth manifold with nonnegative scalar curvature must have nonnegative ADM mass. This result was first proved by Schoen--Yau using minimal surfaces \cite{SchoenYauPMT}. Witten also proved the spin positive mass theorem for all dimensions \cite{MR626707}. It is a natural problem to generalize the positive mass theorem to singular Riemannian manifolds. However, the Schwarzchild metric, which admits a point with a horn singularity, is scalar-flat but has negative mass. In \cite{MR3961306}, Li--Mantoulidis proved a positive mass theorem for AF manifolds with $L^\infty$-metric and fiberwise conical singularities. Dai--Sun--Wang also proved a positive mass theorem for singular spaces with isolated conical singularities \cite{MR4880457,DaiSunWang}.

In this paper, we extend the idea of our main result and prove a positive mass theorem for manifolds with iterated conical singularities.
\begin{theorem}\label{thm:intro-positive mass theorem for iterated asymptotically conical spaces}
    Let \((M^n,g)\) be an n-dimensional asymptotically flat spin manifold with iterated asymptotically conical singularities.
    Suppose that the links of \(M\) have dimension at least two.
    If \(\Sc_g \geq 0\), then \(m(g) \geq 0\).
    Moreover, if \(m(g) = 0\), then the metric \(g\) is flat.
\end{theorem}

This paper is organized as follows.
In Section \ref{sec:manifolds with facs}, we define the notion of a manifold with fiberwise asymptotically conical singularities (FACS), a special case of the notion of a manifold with IACS.
After this, some geometric preliminaries necessary to give a Dirac operator-theoretic proof of our scalar-mean curvature comparison result are presented.
The heart of the proof of Theorem \ref{thm:intro-scalar curvature comparison} lies in the special case of manifolds with FACS. 
Section \ref{sec:fredholmness} establishes the essential self-adjointness of the naturally defined twisted Dirac operator in this special case, from which Fredholmness will follow by a standard argument using compact resolvents.
Section \ref{sec:scalar-mean curvature comparison results} proves Theorem \ref{thm:intro-scalar curvature comparison} for manifolds with FACS using the results of Sections \ref{sec:manifolds with facs} and \ref{sec:fredholmness}.
In Section \ref{sec:manifolds with iacs}, the notion of a manifold with IACS is made precise. Theorem \ref{thm:intro-scalar curvature comparison} will follow from the same proofs given in Sections \ref{sec:fredholmness} and \ref{sec:scalar-mean curvature comparison results}.
Lastly, in Section \ref{sec:applications}, we present some applications of the methods developed throughout the paper. In particular, Theorems \ref{thm:rigidity of NNSC fill-in - intro} and \ref{thm:intro-positive mass theorem for iterated asymptotically conical spaces} are proven.

\vspace{.5cm}
\textbf{Acknowledgments.} We would like to thank Zhizhang Xie for very helpful comments.

\section{Manifolds with fiberwise asymptotically conical singularities}\label{sec:manifolds with facs}
In this section, we make precise the notion of a manifold with fiberwise asymptotically conical singularities (FACS)
and the notion of an asymptotically conical map between manifolds with FACS.
After that, we review some basic facts regarding the geometry of the cone of a Riemannian manifold.
Then, the geometry of an asymptotic family of cones over a base is compared to the geometry of just a regular family of cones over the same base.
In particular, the asymptotic behavior of the error between the scalar curvatures is computed in Proposition \ref{prop:asymptotic scalar curvature estimate for family of cones}.
We then present an explicit relationship between the twisted Dirac operator on a cone and the twisted Dirac operator on its link.
Finally, some standard Listing-type estimates are established.

\subsection{Geometric setup} 

%Let \(M\) be a topological space.
\begin{definition}\label{def:FACS manifolds}
    Let \(M\) be a topological space.
    We call a neighborhood \(U\) a \textit{conical neighborhood} of \(p \in M\) if \(U\) is homeomorphic to a direct product of \(\R^\hdim\) with the cone \(\cone[\varepsilon]{L}\), where \(L\) is a smooth manifold without boundary, for which \(p\) is taken to a point of \(\R^\hdim \times \{0\}\) under the homeomorphism.\footnote{By definition, \(\cone[\varepsilon]{L} = [0, \varepsilon) \times L / \sim\), where \(\sim\) corresponds to crushing \(\{0\} \times L\) to a point.}
    The point \(p\) is a called a \textit{singular point} of \(M\).

    A topological space \(M\) is said to be an \textit{\(\tdim\)-dimensional manifold with fiberwise asymptotically conical singularities} if the following conditions hold:
    \begin{enumerate}[label=(\alph*)]
        \item \emph{Topology:} There is an open cover of \(M\) for which every open set in the cover is either homeomorphic to \(\R^\tdim\) or is a conical neighborhood of some singular point of \(M\).
        The transition functions between two nontrivially intersecting neighborhoods \(U\) and \(V\) both homeomorphic to \(\R^\tdim\) are required to be smooth.
        The transition functions between two nontrivially intersecting conical neighborhoods \(U \cong \R^\hdim \times \cone[\varepsilon]{L}\) and \(V \cong \R^\hdim \times \cone[\varepsilon']{L'}\) are required to take the form:
        \begin{equation*}
            F(b,r,x) = \big(\psi(b), \varphi(b,r), f(b,x)\big)
        \end{equation*}
        where \(b\) represents a point on the base space, \(r\) corresponds to the radial coordinate, and \(x\) is an element of the link.
        The functions \(\psi\), \(\varphi\), and \(f\) are required to be smooth away from \(r=0\), and \(\varphi\) is assumed to take the form \(\varphi(b,r) = a(b) \cdot r + o_2(r)\) near \(r=0\), with \(a(b) > 0\) for all \(b\).\footnote{\(\alpha(r) = o_2(r)\) means that \(\alpha(r) = o(r)\), \(\alpha'(r) = o(1)\), and \(\alpha''(r) = o(1/r)\).}

        The \textit{singular stratum}, also sometimes referred to as the \(\cat{C}_1\)-stratum (see Section \ref{sec:manifolds with iacs}), is the set of singular points of \(M\).
        The singular stratum is denoted by \(\cat{C}\).
        The smooth stratum is defined to be the complement \(M \setminus \cat{C}\).
        \item \emph{Geometry:} There is a metric \(g_\cat{C}\) on \(\cat{C}\) and a metric \(g\) defined on the smooth stratum of \(M\) such that in a conical neighborhood \(U = \R^\hdim \times \cone[\varepsilon]{L}\) of a singular point \(p\) the metric \(g\) takes the form
        \begin{equation*}
            g|_U = \pi^*g_{\cat{C}}|_{\pi(U)} \oplus \beta'(b,r)^2 dr^2 \oplus \beta(b,r)^2 g_L(b) + o_2(r^2)
        \end{equation*}
        subject to the following criteria:
        \begin{enumerate}[label=(\arabic*)]
            \item \(\pi^*g_{\cat{C}}|_{\pi(U)}\) is the pullback of the metric on the singular stratum under the projection \(\pi \colon U \to \R^\hdim\).
            \item \(g_L(b)\) is a family of metrics on the link varying over the base.
            \item \(\beta(b,r)\) is a real-valued function that is smooth away from \(r=0\) and is of the form \(\beta(b,r) = a(b) \cdot r + o_2(r)\) near \(r=0\), with \(a(b) > 0\) for all \(b\).
        \end{enumerate}
        To ease the notation, we will often write \(\pi^*g_B\) to denote the pullback of the base metric, instead of \(\pi^*g_{\cat{C}}|_{\pi(U)}\).
    \end{enumerate}
\end{definition}
The phrase ``manifold with fiberwise asymptotically conical singularities" is sometimes shortened to ``manifold with FACS" or just ``FACS manifold".

Some remarks regarding Definition \ref{def:FACS manifolds} are in order.
First, observe that if \(g\) decomposes as
\begin{equation*}
    g|_U = \pi^*g_B \oplus \beta'(b,r)^2 dr^2 \oplus \beta(b,r)^2 g_L(b) + o_2(r^2)
\end{equation*}
in a conical neighborhood \(U\) of a singular point, then after reparametrizing the radial coordinate to \(\tilde{r} = \beta(b,r)\), the metric becomes
\begin{equation*}
    g|_U = \pi^*g_B \oplus d\tilde{r}^2 \oplus \tilde{r}^2 g_L(b) + o_2(\tilde{r}^2).
\end{equation*}
Thus, after a reparametrization, the metric corresponds to a Riemannian submersion of a family of cones over the base space, up to asymptotic error terms.
Second, for a singular point \(p \in M\) with conical neighborhood \(U \cong \R^\hdim \times \cone[\varepsilon]{L}\), the smooth manifold \(L\) is called the \textit{link of \(p\)}.
It is not hard to see that if \(\mathcal{C}\) is connected, every singular point will have the same link.
Hence, the notion of a link pertains to the connected components of \(\mathcal{C}\), rather than to any particular singular point.
We shall use this fact without explicit mention in the sequel to simplify the exposition.
Third, the definition of a FACS manifold can be naturally extended to define the notion of a \textit{FACS manifold with FACS boundary} by allowing the conical neighborhoods to have base \(B = \R^\hdim\) or \(B = \mathbb{H}^\hdim\).\footnote{\(\mathbb{H}^\hdim\) is the \(\hdim\)-dimensional half-space}
Finally, observe that the singular stratum depends on the covering atlas.
Recall that \(\cone[\varepsilon]{\sphe^\ldim}\), where \(\sphe^\ldim\) is the \(l\)-dimensional sphere, is just the polar representation of a ball of radius \(\varepsilon\) in \(\R^{\ldim+1}\).
Thus, when the links are spheres, \(M\) can be a smooth manifold even though the so-called ``singular stratum" \(\cat{C}\) is nonempty.

\begin{definition}
    If \(\dom{M}\) and \(M\) are manifolds with FACS, a map \(F\) is said to be asymptotically conical if
    \begin{enumerate}
        \item \(F\) maps smooth points to smooth points and singular points to singular points.
        \item \(F\) is smooth when restricted to the smooth stratum.
        \item If \(F\) maps a conical neighborhood \(\dom{U} \cong \R^{\dom{\hdim}} \times \cone[\dom{\varepsilon}]{\dom{L}}\) into a conical neighborhood \(U \cong \R^\hdim \times \cone[\varepsilon]{L}\), then, \(\dom{\hdim} = \dom{h}\), \(\dim \dom{L} = \dim L\), and in the trivialization \(F\) takes the form
            \begin{equation*}
                F(\dom{b},r,\dom{x}) = \big(\psi(\dom{b}), \varphi(\dom{b},r), f(\dom{b},\dom{x})\big)
            \end{equation*}
        where \(\dom{b}\) represents a point on the base space, \(r\) corresponds to the radial coordinate, and \(\dom{x}\) is an element of the link.
        The functions \(\psi\), \(\varphi\), and \(f\) are required to be smooth away from \(r=0\) and \(\varphi\) is assumed to take the form \(\varphi(\dom{b},r) = a(\dom{b}) \cdot r + o_2(r)\) near \(r=0\), where \(a(\dom{b}) > 0\) for all \(\dom{b}\).
    \end{enumerate}
\end{definition}

We now state the scalar-mean curvature comparison theorem for manifolds with FACS, a special case of Theorem \ref{thm:intro-scalar curvature comparison}.

\begin{theorem}\label{thm:scalar-curvature comparison - manifolds with facs}
    Let \((\dom{M}^{\tdim},\bnd{\dom{M}}^{\tdim-1},\dom{g})\) and \((M^{\tdim},\bnd{M}^{\tdim-1},g)\) be compact spin manifolds with FACS and smooth boundaries \(\bnd{\dom{M}}\), \(\bnd{M}\).
    Assume that the links of \(\dom{M}\) and \(M\) have dimension at least two.
    Suppose that the curvature operator of \(M\) and the second fundamental form on the boundary of \(M\) are nonnegative.
    Let \(F \colon \dom{M} \to M\) be an asymptotically conical map, with \(F_\partial\) denoting its restriction to the boundary.
    Assume that one of the following conditions holds:
	\begin{enumerate}[label=\((\roman*)\)]
		\item The boundaries of \(\dom{M}\) and \(M\) coincide, namely, \(\bnd{\dom{M}} = \bnd{M}\) and \(g|_{\bnd{\dom{M}}}\) equals \(g|_{\bnd{M}}\).
		\item The Euler characteristic of \(\bnd{M}\) is nonzero and \(F\) has nonzero degree.
	\end{enumerate}
	If the following comparison conditions hold:
	\begin{enumerate}[label=\textnormal{(\arabic*)}]
		\item \(\Sc_{\dom{g}} \geq \onorm{dF}^2 F^*\Sc_g\)
		\item \(H_{\dom{g}} \geq \onorm{d(F_\partial)} (F_\partial)^*H_g\)
	\end{enumerate}
	then, we conclude
	\begin{enumerate}[label=\textup{(\Roman*)}]
		\item \(\Sc_{\dom{g}} = \onorm{dF}^2 F^* \Sc_g\)
		\item \(H_{\dom{g}} = \onorm{d(F_\partial)} (F_\partial)^*H_g\)
	\end{enumerate}
	Moreover,
	\begin{enumerate}[label=\((\mathrm{\Roman*}^\prime)\)]
		\item If \(\Ric_g > 0\), then \(\onorm{dF} \equiv a\) for some constant \(a > 0\) and \(F \colon (\dom{M}, a \cdot \dom{g}) \to (M,g)\) is a Riemannian covering map.
		\item If \((M, g)\) is flat, then \((\dom{M}, \dom{g})\) is Ricci flat.
	\end{enumerate}
\end{theorem}
\begin{remark}\label{rem:counterexample to comparison theorem when 1-dimensional links are present}
    The above theorem is false if 1-dimensional links are allowed.
    Consider
    \begin{align*}
        \dom{M} &= [0,1] \times \sphe^1 / \sim \text{ with } \dom{g} = dr^2 + r^2 g_{\sphe^1} \\
        M &= [0,a] \times \sphe^1 / \sim \text{ with } g=dr^2 + \frac{1}{a^2}r^2 g_{\sphe^1}
    \end{align*}
    Let \(F(r,x) = (ar,x)\).
    Then the restriction of \(F\) to the boundary is an isometry, the scalar curvatures of both metrics are identically zero, but \(H_{\dom{g}} = 1\) and \(H_g = \frac{1}{a}\).
    This gives a counterexample to Theorem \ref{thm:scalar-curvature comparison - manifolds with facs} as along as \(a > 1\).
\end{remark}
The above theorem holds not only for compact spin manifolds with FACS, but also for so called compact spin \(\cat{C}_m\)-manifolds.
This is discussed in Section \ref{sec:manifolds with iacs}.

\subsection{Geometry of the cone}
Let \((L^{\tdim-1},g)\) be a Riemannian manifold with Levi-Civita connection \(\conn[L]\) and curvature tensor \(\curvtens[L]\).
The curvature operator of \((L,g)\) is denoted \(\curvoper[L]\) and is defined pointwise on \(\bigwedge^2TM\) by the equation
\begin{equation*}
    \inner{\curvoper[L](e_i\wedge e_j),e_k\wedge e_l}=-\inner{\curvtens[L]_{e_i,e_j}e_k,e_l}.
\end{equation*}
The sign is chosen so that all the sectional curvatures are nonnegative if \(\curvoper[L]\) is a nonnegative operator.
The scalar curvature is then given by
\begin{equation*}
    \Sc_L=\sum_{i,j}\left<\curvtens[L]_{e_i,e_j}e_j,e_i\right>=2\trac\curvoper[L]
\end{equation*}

Let \((CL^{\tdim},g_{CL})\) denote the cone of \((L,g)\).
Topologically, \(CL\) is just the product \(CL = (0,\infty)\times L\).\footnote{Throughout this paper, \(\cone{L}\) is used to denote both the cone with the singular point included as well as the cone take away the singular point. It should be clear from the context to which we refer.} The metric is given by \(g_{CL}=dr^2 + r^2g_L\).
Of course, \(L\) isometrically sits inside its cone with \(L\simeq \{1\}\times L \subseteq CL\).
If \(L\) is oriented, then we endow \(CL\) with the product orientation on \((0, \infty) \times L\).

Given an orthonormal frame \(\{e_1,\ldots,e_n\}\) on \(L\) defined in an open neighborhood \(U\), it naturally induces an orthonormal frame on \((0,\infty)\times U=\cone{U}\subseteq \cone{L}\) given by \(\{\partial_r,\frac{e_1}{r},\ldots,\frac{e_n}{r}\}\).
Using this association, we can write the connection \(\conn[CL]\) in terms of the connection \(\conn[L]\).
\begin{proposition}[cf. {\cite[Lemma 2.2]{Chou:1985}}]\label{prop:covariant derivative on cone}
    \begin{align*}
        \left< \conn[CL]_{\frac{e_i}{r}}\frac{e_j}{r},\frac{e_k}{r} \right>_{CL}&=\frac{1}{r}\left< \conn[L]_{e_i}e_j,e_k \right>_{L}\\
        \left< \conn[CL]_{\frac{e_i}{r}}\frac{e_j}{r}, \partial_r \right>_{CL}&=-\frac{1}{r}\delta_{ij}\\
        \left< \conn[CL]_{\frac{e_i}{r}}\partial_r, \frac{e_j}{r} \right>_{CL}&=\frac{1}{r}\delta_{ij}\\
        \left< \conn[CL]_{\frac{e_i}{r}}\partial_r, \partial_r \right>_{CL}&=0
    \end{align*}
\end{proposition}
\begin{proof}
    This follows from the Koszul formula:
    \begin{multline*}
        2\left< \nabla_X Y,Z \right>=X\left< Y,Z \right> + Y\left< X,Z \right> - Z\left< X,Y \right>\\ + \left< [X,Y],Z \right> - \left< [X,Z],Y \right> - \left< [Y,Z],X \right>.
    \end{multline*}
\end{proof}
In terms of the connection matrix, the above proposition states that if \((\omega_j^i)\) is the connection matrix on \(U\) with respect to \(\{e_1,\ldots,e_n\}\), then the connection matrix on \(CU\) with respect to \(\{\partial_r,\frac{e_1}{r},\ldots,\frac{e_n}{r}\}\) is
\begin{equation}\label{eqn:connection matrix of cone}
    \begin{pmatrix}
        0 & \begin{matrix} -e_1^* & \ldots & -e_n^*\end{matrix}\\
        \begin{matrix}
            e_1^*\\
            \vdots\\
            e_n^*
        \end{matrix} & \omega_j^i
    \end{pmatrix}
\end{equation}
where \(\{e_i^*\}\) are the differential forms for which \(e_i^*(e_j)=\delta_{ij}\).

For each \(r\), the hypersurface \(\Sigma_r=\{r\}\times L\) in \(CL\) has unit normal \(\partial_r\).
The cone metric restricted to \(\Sigma_r\) is \(r^2g_L\), in particular, it is a scalar multiple of \(g_L\). The Levi-Civita connection is invariant under scalar multiples, so \(\conn[\Sigma_r]=\conn[L]\).
Thus, the (3,1)-curvature tensors are equal, \(\curvtens[\Sigma_r]=\curvtens[L]\).
This, combined with the Gauss equation, allows us to compute the curvature of \(CL\).
Recall that the Gauss equation states that for \(X,Y,Z,W\in T\Sigma_r=TL\),
\begin{multline}\label{Gauss equation}
    \curvtens[CL](X,Y,Z,W) = \curvtens[\Sigma_r](X,Y,Z,W) + \left< \mathbf{A}(X,Z), \mathbf{A}(Y,W) \right>_{CL}\\ - \left< \mathbf{A}(X,W), \mathbf{A}(Y,Z) \right>_{CL},
\end{multline}
where \(\mathbf{A}(X,Y)\coloneqq\left< \nabla_X^{CL}Y, \partial_r \right>_{CL}\partial_r\) for \(X,Y\in T\Sigma_r\).
Now,
\begin{align*}
    \left< \nabla_X^{CL}Y,\partial_r \right>_{CL} &= -\left< Y,\nabla_X^{CL}\partial_r \right>_{CL}\\
    &= -\frac{1}{r} \left< Y,X \right> _{CL}\\
    &= -r\left< X,Y \right>_L.
\end{align*}
At the same time, if \(\curvtens[\Sigma_r]=\curvtens[L]\) as (3,1)-tensors, then as (4,0)-tensors
\begin{equation*}
    \curvtens[\Sigma_r](X,Y,Z,W)=r^2\curvtens[L](X,Y,Z,W).
\end{equation*}
Thus, Equation \eqref{Gauss equation} becomes
\begin{multline*}
    \curvtens[CL](X,Y,Z,W) = r^2\curvtens[L](X,Y,Z,W) + r^2\left< X,Z \right>_L\left< Y,W \right>_L \\
    - r^2\left< X,W \right>_L\left< Y,Z \right>_L.
\end{multline*}
Since $\left< X\wedge Y, Z\wedge W \right>_L = \begin{vmatrix}
    \left< X,Z \right>_L & \left< X,W \right>_L \\
    \left< Y,Z \right>_L & \left< Y,W \right>_L
\end{vmatrix}$, the last equation can be rewritten as
\begin{align*}
    \left< \curvoper[CL](X\wedge Y),Z\wedge W \right>_{CL} &= r^2\left< \curvoper[L](X\wedge Y) - X \wedge Y, Z\wedge W \right>_L \\
    &= \frac{1}{r^2}\left< \curvoper[L](X \wedge Y) - X \wedge Y, Z \wedge W \right>_{CL}.
\end{align*}

It is also not hard to compute, using Proposition \ref{prop:covariant derivative on cone}, that \(\curvoper[CL](\partial_r\wedge X)=0\) for \(X \in T\Sigma_r\).
We have shown
\begin{proposition}\label{prop:curvature of cone}
    Let \((L^{\tdim-1},g)\) be a Riemannian manifold and \((CL^{\tdim},g_{CL})\) be its cone.
    Let \(\curvoper[L]\) and \(\curvoper[CL]\) denote their curvature operators.
    Then
    \begin{align*}
        \curvoper[CL]|_{\bigwedge^2 TL} &= \frac{1}{r^2} \left(\curvoper[L] - Id|_{\bigwedge^2 TL}\right) \\
        \curvoper[CL](\partial_r\wedge X)&=0 \text{ for } X\in TL
    \end{align*}
\end{proposition}
\begin{corollary}\label{cor:scalar curvature of cone}
    The scalar curvature of \((CL^{\tdim},g_{CL})\) is given by
    \begin{equation*}
        \Sc_{CL} = \frac{\Sc_L-(n-1)(n-2)}{r^2}.
    \end{equation*}
\end{corollary}

\subsection{Geometry of a family of asymptotic cones}\label{subsec:geometry of family of asymptotic cones}
In this section, we study the geometry of a family of cones parametrized by a base space, having asymptotically conical metrics. The setup is as follows: There exists a Riemannian manifold \((B^{\hdim},g_B)\), representing the base, and a manifold \(L^{\ldim}\), representing the link.
The space in question is the direct product \(M^\tdim = B^\hdim \times (0, \infty) \times L^\ldim\).
Let \(g_L(b)\) be a family of metrics on the link varying smoothly over the base.
Defining \(\bund{L} \coloneqq B \times L\) with \(g_{\bund{L}} = \pi^*g_B \oplus g_L(b)\) yields a Riemannian link bundle over \(B\).
We write \(\cone{\bund{L}}\) to denote the fiberwise coning of \(\bund{L}\).
Namely, \(\cone{\bund{L}} = B^\hdim \times (0, \infty) \times L^\ldim\) with metric
\begin{equation*}
    g_{\bund{L}} = \pi^*g_B \oplus dr^2 \oplus r^2 g_L(b).
\end{equation*}
The metric on \(M\) is then assumed to differ from the metric on \(C\bund{L}\) by \(o_2(r^2)\) error terms.
Thus, it is of the form
\begin{equation*}
    g_M = dr^2 + r^2 g_{\bund{L}} + o_2(r^2).
\end{equation*}

The next proposition relates the scalar curvature of \(M\) to the scalar curvature of \(\cone{\bund{L}}\).
\begin{proposition}\label{prop:asymptotic scalar curvature estimate for family of cones}
    \begin{equation*}
        \Sc_M = \Sc_{\cone{\bund{L}}} + o(1/r^2)
    \end{equation*}
\end{proposition}
\begin{proof}
    Let \(g\) denote the metric on \(\cone{\bund{L}}\) and \(h\) the metric on \(M\).
    Take an orthonormal frame \(\{b_1, \ldots, b_{\hdim}, \partial_r, \frac{e_1}{r}, \ldots, \frac{e_{\ldim}}{r}\} = \{\{b_i\}, \partial_r, \{e_i\}\}\) with respect to \(g\).
    Here, the \(\{e_i\}\) are taken to be in \(TL\) and the \(\{b_i\}\) are lifted from the base.

    By assumption
    \begin{align*}
        \inner{b_i, b_j}_h &= \delta_{ij} + o_2(r^2) &\inner{b_i, \partial_r}_h &= o_2(r^2) & \inner{b_i, \frac{e_i}{r}}_h &= o_2(r) \\
        \inner{\partial_r, b_j}_h &= o_2(r^2) & \inner{\partial_r, \partial_r}_h &= 1 + o_2(r^2) & \inner{\partial_r, \frac{e_i}{r}}_h &= o_2(r) \\
        \inner{\frac{e_i}{r}, b_j}_h &= o_2(r) & \inner{\frac{e_i}{r}, \partial_r}_h &= o_2(r) & \inner{\frac{e_i}{r}, \frac{e_j}{r}}_h &= \delta_{ij} + o_2(1)
    \end{align*}
    Applying Gram-Schmidt orthogonalization to \(\{\{b_i\}, \partial_r, \{e_i\}\}\) with respect to the metric \(h\), we obtain an orthonormal frame \(\{\{\dom{b}_i\}, \dom{\partial}_r, \{\frac{\dom{e}_i}{r}\}\}\).
    Following the Gram-Schmidt algorithm, it is not hard to see that
    \begin{align}
        \label{eqn:new from original - base} \dom{b}_i &= b_i + o_2(r^2)\{b_j\} \\
        \label{eqn:new from original - radial} \dom{\partial_r} &= \partial_r + o_2(r^2) \partial_r + o_2(r^2)\{b_j\} \\
        \label{eqn:new from original - link} \frac{\dom{e}_i}{r} &= \frac{e_i}{r} + o_2(1) \{\frac{e_j}{r}\} + o_2(r) \partial_r + o_2(r) \{b_j\}
    \end{align}
    The converse holds too, namely
    \begin{align}
        \label{eqn:original from new - base} b_i &= \dom{b}_i + o_2(r^2)\{\dom{b}_j\} \\
        \label{eqn:original from new - radial}\partial_r &= \dom{\partial}_r + o_2(r^2) \dom{\partial}_r + o_2(r^2)\{\dom{b}_j\} \\
        \label{eqn:original from new - link} \frac{e_i}{r} &= \frac{\dom{e}_i}{r} + o_2(1) \{\frac{\dom{e}_j}{r}\} + o_2(r) \dom{\partial}_r + o_2(r) \{\dom{b}_j\}
    \end{align}

    Let \(\{v_i\}\) label the elements of \(\{\{b_i\}, \partial_r, \{e_i\}\}\) and \(\{\dom{v}_i\}\) label the elements of \(\{\{\dom{b}_i\}, \dom{\partial}_r, \{\frac{\dom{e}_i}{r}\}\}\).
    Let \(\omega_{ij}^k\) denote \(\inner{\conn[g]_{v_i} v_j, v_k}_g\) and \(\dom{\omega}_{ij}^k\) denote \(\inner{\conn[h]_{\dom{v}_i} \dom{v}_j, \dom{v}_k}_h\).
    Define \(c_{ij}^k = \inner{[v_i, v_j], v_k}_g\) and \(\dom{c}_{ij}^k = \inner{[\dom{v}_i, \dom{v}_j], \dom{v}_k}_h\).
    Applying \eqref{eqn:new from original - base}-\eqref{eqn:new from original - link} and \eqref{eqn:original from new - base}-\eqref{eqn:original from new - link} we see that, at worst,
    \begin{equation}\label{eqn:new lie bracket coefficient}
        \dom{c}_{ij}^k = c_{ij}^k + o_1(1/r)
    \end{equation}
    From the Koszul formula
    \begin{align*}
        \omega_{ij}^k &= \frac{1}{2} (c_{ij}^k - c_{ik}^j - c_{jk}^i) \\
        \dom{\omega}_{ij}^k &= \frac{1}{2} (\dom{c}_{ij}^k - \dom{c}_{ik}^j - \dom{c}_{jk}^i)
    \end{align*}
    and hence, at worst
    \begin{equation}\label{eqn:new connection matrix coefficient}
        \dom{\omega}_{ij}^k = \omega_{ij}^k + o_1(1/r)
    \end{equation}

    The scalar curvature formula, written in terms of indices associated to a nonholonomic orthonormal frame, is given by
    \begin{align*}
        \Sc_{ C\bund{L} } &= \sum_j \left[ v_i (\omega_{j j}^i) - v_j (\omega_{i j}^i) + \omega_{j j}^k \omega_{i k}^i - \omega_{i j}^k \omega_{j k}^i - c_{i j}^k \omega_{k j}^i \right] \\
        \Sc_{M} &= \sum_j \left[ \dom{v}_i (\hat{\omega}_{j j}^i) - \dom{v}_j (\hat{\omega}_{i j}^i) + \hat{\omega}_{j j}^k \hat{\omega}_{i k}^i - \hat{\omega}_{i j}^k \hat{\omega}_{j k}^i - \hat{c}_{i j}^k \hat{\omega}_{k j}^i \right]
    \end{align*}
    where \(v(\omega_{ij}^k)\) denotes differentiation of the real-valued function \(\omega_{ij}^k\) in the direction of \(v\).
    Combining the above with Equations \eqref{eqn:new lie bracket coefficient} and \eqref{eqn:new connection matrix coefficient} implies the conclusion.
\end{proof}

\begin{remark}\label{rmk:asymptotic sectional curvature for family of cones}
    The above proof shows that the error terms in the sectional curvatures are all \(o(1/r^2)\).
\end{remark}

\subsection{Twisted Dirac operator on a cone}\label{subsec:twisted dirac operator on cone}
Let \((\dom{L}^{\tdim-1}, \dom{g})\) and \((L^{\tdim-1}, g)\) be closed spin Riemannian manifolds of the same dimension and suppose \(F \colon \cone{\dom{L}} \to \cone{L}\) is a smooth map of the form \(F(r,x) = (\varphi(r), f(x))\). Define \(V = T\cone{\dom{L}} \oplus F^*T\cone{L}\). We want to relate the Dirac operator on \(SV\) to the Dirac operator on \(SV|_{\dom{L}}\). To do so, we will construct a unitary \(\Psi \colon L^2((0, \infty), SV|_{\dom{L}}) \to L^2(SV)\).

Before constructing \(\Psi\), let us establish some notation.
First, notice that the spinor bundle \(SV\) is a \(\Cl(T\cone{\dom{L}} \oplus F^*T\cone{L})\)-module.
The Clifford action by vectors \(\dom{w}\) coming from \(T\cone{\dom{L}}\) will be denoted by \(\dom{c}(\dom{w})\) and the Clifford action by vectors \(w\) coming from \(F^*T\cone{L}\) will be denoted by \(c(w)\). Second, given \(\sigma_x\) in \((SV|_{\dom{L}})_x\), let \(\overline{\sigma_x}\) denote the parallel transport of \(\sigma_x\) along the ray \((0, \infty) \times \{x\}\).
This parallel transport satisfies the following properties.
\begin{lemma}
    If \(X \in T\dom{L}\) and \(Y \in f^*TL\), then
    \begin{gather}
        \label{eqn:clifford action pulls out of parallel transport - vector from domain} \overline{\dom{c}(X) \sigma} = \dom{c}(\overline{X}) \overline{\sigma} \\
        \label{eqn:clifford action pulls out of parallel transport - vector from target} \overline{c(Y) \sigma} = c(\overline{Y}) \overline{\sigma} \\
        \label{eqn:derivative of parallel transport} \conn[]_X \overline{\sigma} = \overline{\conn[]_X \sigma}
    \end{gather}
    Here, \(\overline{X}\) and \(\overline{Y}\) are the parallel transports of \(X\) and \(Y\) in the radial direction.
\end{lemma}
\begin{proof}
    Equations \eqref{eqn:clifford action pulls out of parallel transport - vector from domain} and \eqref{eqn:clifford action pulls out of parallel transport - vector from target} are an immediate consequence of the compatibility of the spinor connection with the Clifford action.

    To prove Equation \eqref{eqn:derivative of parallel transport}, take an orthonormal frame \(\big\{\partial_r, \{\frac{\dom{e}_i}{r}\}, \varphi^*\partial_r, \{\frac{f^*e_i}{\varphi(r)}\}\big\}\) on \(V\).
    The associated connection matrix takes the form
    \begin{multline*}
        \omega_V = \sum_i \dom{\omega}_{\partial_r}^i \partial_r \wedge \frac{\dom{e}_i}{r} + \sum_{i,j} \dom{\omega}_j^i \frac{\dom{e}_j}{r} \wedge \frac{\dom{e}_i}{r} + \sum_i (F^*\omega_{\partial_r}^i)\, \varphi^*\partial_r \wedge \frac{f^*e_i}{\varphi(r)} \\
        + \sum_{i,j} (F^*\omega_j^i)\, \frac{f^*e_j}{\varphi(r)} \wedge \frac{f^*e_i}{\varphi(r)}
    \end{multline*}
    In the above, \(X \wedge Y\) is the endomorphism defined by \((X \wedge Y)(w) = \inner{X, w} Y - \inner{Y,w} X\).

    Lifting the frame and connection to the spinor bundle, one gets an orthonormal spinor frame \((\sigma_1, \ldots, \sigma_N)\) for which
    \begin{multline*}
        \nabla \sigma_k = \sum_i \frac{1}{2}\dom{\omega}_{\partial_r}^i \cbar{\partial_r} \cbar{\frac{\dom{e}_i}{r}} \sigma_k + \sum_{i,j} \frac{1}{2}\dom{\omega}_j^i \cbar{\frac{\dom{e}_j}{r}} \cbar{\frac{\dom{e}_i}{r}} \sigma_k \\
        + \sum_i \frac{1}{2} (F^*\omega_{\partial_r}^i) c(\varphi^* \partial_r) c\Big(\frac{f^*e_i}{\varphi(r)}\Big) \sigma_k + \sum_{i,j} \frac{1}{2} (F^*\omega_j^i) c\Big(\frac{f^*e_j}{\varphi(r)}\Big) c\Big(\frac{f^*e_i}{\varphi(r)}\Big) \sigma_k.
    \end{multline*}
    Observe that the frame \(\{\sigma_i\}\) is parallel with respect to the radial direction, so actually \((\sigma_1, \ldots, \sigma_N) = (\overline{\sigma}_1, \ldots, \overline{\sigma}_N)\).

    Thus,
    \begin{align*}
        \conn[]_X \overline{\sigma}_k &= \sum_i \frac{1}{2}\dom{\omega}_{\partial_r}^i(X) \cbar{\partial_r} \cbar{\frac{\dom{e}_i}{r}} \overline{\sigma}_k + \sum_{i,j} \frac{1}{2}\dom{\omega}_j^i(X) \cbar{\frac{\dom{e}_j}{r}} \cbar{\frac{\dom{e}_i}{r}} \overline{\sigma}_k \\
        &+ \sum_i \frac{1}{2} \omega_{\partial_r}^i(f_*X) c(\varphi^* \partial_r) c\Big(\frac{f^*e_i}{\varphi(r)}\Big) \overline{\sigma}_k + \sum_{i,j} \frac{1}{2} \omega_j^i(f_*X) c\Big(\frac{f^*e_j}{\varphi(r)}\Big) c\Big(\frac{f^*e_i}{\varphi(r)}\Big) \overline{\sigma}_k\\
        &= \overline{\sum_i \frac{1}{2} \dom{\omega}_{\partial_r}^i(X) \cbar{\partial_r} \cbar{\dom{e}_i} \sigma_k + \sum_{i,j} \frac{1}{2}\dom{\omega}_j^i(X) \cbar{\dom{e}_j} \cbar{\dom{e}_i} \sigma_k} \\
        &+ \overline{\sum_i \frac{1}{2} \omega_{\partial_r}^i(f_*X) c(\varphi^* \partial_r) c\Big(\frac{f^*e_i}{\varphi(1)}\Big) \sigma_k + \sum_{i,j} \frac{1}{2} \omega_j^i(f_*X) c\Big(\frac{f^*e_j}{\varphi(1)}\Big) c\Big(\frac{f^*e_i}{\varphi(1)}\Big) \sigma_k}\\
        &= \overline{\conn[]_X \sigma_k},
    \end{align*}
    where in the above we have used that the connection matrices \(\dom{\omega}\) and \(\omega\) are constant with respect to the radial direction \(r\) (see \eqref{eqn:connection matrix of cone}).

    The general case where \(\sigma\) in \(SV|_{\dom{L}}\) is a linear combination of the \(\{\sigma_i\}\) follows from the above.
\end{proof}

The unitary \(\Psi\) is constructed as follows.
Given \(\sigma(r)\) in \(C^{\infty}((0,\infty), SV|_{\dom{L}})\), define \((\Psi \sigma)(r) = r^{-\frac{n-1}{2}}\overline{\sigma(r)}(r)\), where recall that \(n-1\) is the dimension of the link.
In words, \((\Psi \sigma)(r)\) is the parallel transport of \(\sigma(r)\) to the \(r\)-link, \(\{r\} \times \dom{L}\), multiplied by \(r^{-\frac{n-1}{2}}\).
The \(r^{-\frac{n-1}{2}}\) term ensures \(\Psi\) is a unitary.
To see this, we compute
\begin{align*}
    \norm{\Psi \sigma}_{L^2}^2 &= \int_{\dom{L}} \int_{0}^{\infty} \pnorm{r^{-\frac{n-1}{2}} \overline{\sigma}}^2 r^{n-1} dr\\
    &= \int_{\dom{L}} \int_{0}^{\infty} \pnorm{\sigma}^2 dr\\
    &= \norm{\sigma}_{L^2}^2.
\end{align*}
Hence, \(\Psi\) extends to a unitary \(L^2((0, \infty), SV|_{\dom{L}}) \to L^2(SV)\).

The next step is to explicitly determine the form of the Dirac operator on \(SV\) when viewed as an unbounded operator acting on \(L^2((0, \infty), SV|_{\dom{L}})\).
To this end, start by defining
\begin{equation*}
    T_X^{\partial, 0} \sigma = \frac{1}{2} \sum_i \dom{\omega}_{\partial_r}^i(X) \cbar{\partial_r} \cbar{\frac{\dom{e}_i}{r}} \sigma
\end{equation*}
and consider the new connection \(\conn[\partial, 0]_X \sigma \coloneqq \conn[]_X \sigma - T_X^{\partial, 0} \sigma\).
Intuitively, we are just removing the contribution from the radial direction of \(T\cone{\dom{L}}\).
The next proposition expresses \(\Psi^{-1} D \Psi\) in terms of \(\conn[\partial,0]\).
\begin{proposition}[cf. {\cite[Proposition 5.2]{Lesch:1993}}]\label{prop:dirac operator under unitary transformation}
    Suppose \(\sigma \in C^\infty((0, \infty), SV|{\dom{L}})\) is a smooth compactly supported section vanishing near the singularity.
    Then
    \begin{equation*}
        \Psi^{-1}D\Psi \sigma = \cbar{\partial_r} \frac{d}{dr} \sigma + \frac{1}{r} \sum_i \cbar{\dom{e}_i} \conn[\partial, 0]_{\dom{e}_i} \sigma
    \end{equation*}
    where \(\{\dom{e}_i\}\) is any local orthonormal frame on \(\dom{L}\).
\end{proposition}
\begin{proof}
    To prove the statement, it is sufficient to show it for sections \(\sigma\) supported in \((0, \infty) \times U\), for any coordinate neighborhood \(U\) of the link.
    In this case, \(\sigma\) can be written as a linear combination \(\sigma = \sum_k a_k(r,x) \sigma_k\), where each \(\sigma_k\) depends only on \(x\).
    Thus, it is enough to prove the statement for sections of the form \(a(r,x) \sigma\), where \(a(r,x)\) is a smooth function and \(\sigma\) is a smooth section of \(SV|_{\dom{L}}\).
    We compute
    \begin{align*}
        D\Psi a(r,x) \sigma &= D(r^{-\frac{n-1}{2}} a(r,x) \overline{\sigma}(r))\\
        &= \cbar{\partial_r} \conn[]_{\partial_r} (r^{-\frac{n-1}{2}} a(r,x) \overline{\sigma}(r)) + \sum_i \cbar{\frac{\dom{e}_i}{r}} \conn[]_{\frac{\dom{e}_i}{r}} (r^{-\frac{n-1}{2}} a(r,x) \overline{\sigma}(r))\\
        &=\cbar{\partial_r} \left[ -\Big( \frac{n-1}{2} \Big) r^{-\frac{n-1}{2}} \frac{1}{r} a(r,x) \overline{\sigma}(r) + r^{-\frac{n-1}{2}} a'(r,x) \overline{\sigma}(r) \right] \\
        &\qquad\qquad\qquad\qquad + \sum_i r^{-\frac{n-1}{2}} \frac{1}{r} \overline{\cbar{\dom{e}_i} \conn[]_{\dom{e}_i} a(r,x) \sigma}(r)\\
        &= \cbar{\partial_r} \left[ -\Big( \frac{n-1}{2} \Big) r^{-\frac{n-1}{2}} \frac{1}{r} a(r,x) \overline{\sigma}(r) + r^{-\frac{n-1}{2}} a'(r,x) \overline{\sigma}(r) \right] \\
        &\qquad\qquad\qquad\qquad + \sum_i r^{-\frac{n-1}{2}} \frac{1}{r} \overline{\cbar{\dom{e}_i} \conn[\partial, 0]_{\dom{e}_i} a(r,x) \sigma}(r)\\
        &\qquad\qquad\qquad\qquad + \sum_{i,j} r^{-\frac{n-1}{2}} \frac{1}{r} a(r,x) \overline{\cbar{\dom{e}_i} \frac{1}{2}\dom{\omega}_{\partial_r}^j(\dom{e}_i) \cbar{\partial_r} \cbar{\dom{e}_j} \sigma}(r).
    \end{align*}
    On the other hand, \(\omega_{\partial_r}^j (\dom{e}_i) = \binner{\conn[]_{\dom{e}_i} \partial_r, \frac{\dom{e}_j}{r}} = \delta_{ij}\).
    Thus,
    \begin{align*}
        D\Psi a(r,x) \sigma &= \cbar{\partial_r} \left[ \underbrace{-\Big( \frac{n-1}{2} \Big) r^{-\frac{n-1}{2}} \frac{1}{r} a(r,x) \overline{\sigma}(r)} + r^{-\frac{n-1}{2}} a'(r,x) \overline{\sigma}(r) \right] \\
        &\qquad\qquad\qquad\qquad + \sum_i r^{-\frac{n-1}{2}} \frac{1}{r} \overline{\cbar{\dom{e}_i} \conn[\partial, 0]_{\dom{e}_i} a(r,x) \sigma}(r) \\
        &\qquad\qquad\qquad\qquad + \underbrace{\sum_i r^{-\frac{n-1}{2}} \frac{1}{r} a(r,x) \overline{\cbar{\partial_r} \sigma}(r)}.
    \end{align*}
    The underbraced terms cancel out, which implies
    \begin{equation*}
        \Psi^{-1}D\Psi a(r,x) \sigma = \cbar{\partial_r} a'(r,x) \sigma + \frac{1}{r} \sum_i \cbar{\dom{e}_i} \conn[\partial, 0]_{\dom{e}_i} a(r,x) \sigma.
    \end{equation*}
\end{proof}
\begin{remark}
    The bundle \(SV|_{\dom{L}} = S(T\cone{\dom{L}} \oplus F^*T\cone{L})|_{\dom{L}}\) can be viewed as a module over \(\Cl(T\dom{L} \oplus (F^*T\cone{L})|_{\dom{L}})\) as follows: Given \(\sigma \in SV|_{\dom{L}}\), if \(X \in T\dom{L}\), define \(\dom{c}_\partial(X) \sigma = \cbar{\partial_r} \cbar{X} \sigma\).
    If \(Y \in (F^*TCL)|_{\dom{L}}\), define \(c_\partial(Y) \sigma = \cbar{\partial_r} c(Y) \sigma\).
    Viewing \(SV|_{\dom{L}}\) as a spinor bundle associated to \(T\dom{L} \oplus (F^*T\cone{L})|_{\dom{L}}\), the induced spinor connection is precisely \(\conn[\partial, 0]\).
    In this sense, Proposition \ref{prop:dirac operator under unitary transformation} expresses the Dirac operator \(SV\) in terms of the Dirac operator on \(SV|_{\dom{L}}\).
\end{remark}

To finish the section, recall that \(\dom{L}^{\tdim-1}\) and \(L^{\tdim-1}\) have the same dimension, which means that \(V = T\cone{\dom{L}} \oplus F^*T\cone{L}\) is even-dimensional.
Hence, \(SV\) decomposes into \(SV = S^+V \oplus S^-V\). Notice also that the positive component is canonically isomorphic to the negative component by the map which takes \(\sigma \in S^+V\) to \(\cbar{\partial_r} \sigma \in S^-V\).
Taking into account the \(\mathbb{Z}_2\)-grading along with this identification, we construct a unitary \(\widetilde{\Psi}\) from \(L^2((0, \infty), S^+V|_{\dom{L}} \oplus S^+V|_{\dom{L}})\) to \(L^2(SV)\).
Proposition \ref{prop:dirac operator under unitary transformation} can be rewritten as follows:
\begin{proposition}\label{prop:dirac operator under unitary transformation with grading}
    Suppose \(\sigma\) is a smooth compactly supported section vanishing near the singularity.
    Then
    \begin{equation*}
        \widetilde{\Psi}^{-1}D\widetilde{\Psi} \sigma = \begin{pmatrix}
            0 & -1\\
            1 & 0
        \end{pmatrix} \frac{d}{dr} \sigma + \frac{1}{r}\begin{pmatrix}
            0 & \mathcal{P}\\
            \mathcal{P} & 0
        \end{pmatrix} \sigma
    \end{equation*}
    where \(\mathcal{P} = -\sum_i \cbar{\partial_r} \cbar{\dom{e}_i} \conn[\partial, 0]_{\dom{e}_i}\), with \(\{\dom{e}_i\}\) any local orthonormal frame on \(\dom{L}\).
\end{proposition}
\begin{remark}\label{rem:link operator in terms of link dirac operator}
    Let \(T^{0,\partial}\) be the endomorphism-valued 1-form defined by
    \begin{equation*}
    T_X^{0,\partial} \sigma = \frac{1}{2} \sum_i (F^*\omega_{\partial_r}^i)(X) c(\varphi^* \partial_r) c\Big(\frac{f^*e_i}{\varphi(r)}\Big).\end{equation*}
    \(SV|_{\dom{L}}\) can be viewed as a \(\Cl(T\dom{L} \oplus f^*TL)\)-module by defining \(\dom{c}_\partial(X) = \dom{c}(\partial_r)\dom{c}(X)\) for \(X \in TL\) and \(c_\partial(Y) = \dom{c}(\partial_r)c(\frac{Y}{\varphi(r)})\) for \(Y \in f^*TL\).
    The induced spinor connection is given by \(\conn[\partial, \partial] = \conn[] - T^{\partial, 0} - T^{0, \partial}\).
    \(\mathcal{P}\) can be rewritten as
    \begin{equation*}
        \mathcal{P} = -\sum_i \dom  c_\partial(\dom{e}_i) \conn[\partial, \partial]_{\dom{e}_i} + \dom{c}_\partial(\dom{e}_i) T^{0, \partial}_{\dom{e}_i}.
    \end{equation*}
    Under the above identifications, \(\sum_i \dom  c_\partial(\dom{e}_i) \conn[\partial, \partial]\) corresponds to the Dirac operator on \(S(T\dom{L} \oplus f^*TL)\).
    In particular, \(\mathcal{P}\) differs from the negative of the Dirac operator on \(S(T\dom{L} \oplus f^*TL)\) by a bounded endomorphism.
\end{remark}

\subsection{Dirac operator estimates}
In this subsection, we present some standard Listing-type estimates relating the twisted Dirac operator with the scalar and mean curvatures on the domain and target.
\begin{theorem}\label{thm:scalar curvature estimate}
    Let \((\dom{M}, \bnd{\dom{M}}, \dom{g})\) and \((M, \bnd{M}, g)\) be manifolds with FACS and \(F\) an asymptotically conical map between them.
    If the curvature operator \(\mathcal{R}\) on M is nonnegative and if \(\sigma\) is a smooth section of \(S(T\dom{M} \oplus F^*TM)\) vanishing near the boundary and near the singular stratum of \(\dom{M}\), then the following inequality holds:
    \begin{equation}\label{eqn:scalar curvature estimate}
        \int_{\dom{M}}\pnorm{D\sigma}^2\geq \int_{\dom{M}}\pnorm{\nabla \sigma}^2+
        \int_{\dom{M}} \frac{1}{4}(\dom{\Sc}-\onorm{\wedge^2dF}F^*\Sc )\pnorm{\sigma}^2.
    \end{equation}
\end{theorem}
\begin{proof}
    Since \(\sigma\) vanishes near the boundary,
    \begin{equation*}
        \int_{\dom{M}}\pnorm{D\sigma}^2=\int_{\dom{M}}\langle D^2\sigma,\sigma\rangle.
    \end{equation*}
    By the twisted Lichnerowicz theorem,
    \begin{equation*}
        D^2=\nabla^*\nabla+\frac{\dom{\Sc}}{4}-\frac{1}{2}\sum_{ \substack{i<j \\ k<l} }\left<\mathcal{R}(F_*\dom{e}_i\wedge F_*\dom{e}_j),e_k\wedge e_l\right>\dom{c}(\dom{e}_i)\dom{c}(\dom{e}_j)c(e_k)c(e_l).
    \end{equation*}
    The only term that needs to be investigated is the last one.
    For the sake of simplifying notation, enumerate \(\{\dom{e}_i\wedge\dom{e}_j\}_{i<j}\) by \(\{ \dom{w}_1,\dom{w}_2,\ldots\}\) and \(\{e_i\wedge e_j\}_{i<j}\) by \(\{w_1,w_2,\ldots\}\).
    Observe that writing \(\dom{c}(\dom{w}_k)\) makes sense if we define \(\dom{c}(\dom{e}_i\wedge\dom{e}_j)=\dom{c}(\dom{e}_i)\dom{c}(\dom{e}_j)\), and the same goes for \(c(w_k)\).
    Under this notation, the last term can be written as \(-\frac{1}{2}\sum_{i,j}\left<\mathcal{R}F_*\dom{w}_i,w_j\right>\dom{c}(\dom{w}_i)c(w_j)\).
    Since \(\mathcal{R}\) is nonnegative, it has a square root \(L\).
    We compute
    \begin{equation*}
        \begin{aligned}
            -\frac{1}{2}\sum_{i,j}\left<\mathcal{R}F_*\dom{w}_i,w_j\right>\dom{c}(\dom{w}_i)c(w_j)&=-\frac{1}{2}\sum_{i,j}\langle LF_*\dom{w}_i,Lw_j\rangle\dom{c}(\dom{w}_i)c(w_j)\\
            &=-\frac{1}{2}\sum_{i,j,k}\langle LF_*\dom{w}_i,w_k\rangle\langle Lw_j,w_k\rangle\dom{c}(\dom{w}_i)c(w_j)\\
            &=-\frac{1}{2}\sum_{i,j,k}\langle F_*\dom{w}_i,Lw_k\rangle\langle w_j,Lw_k\rangle\dom{c}(\dom{w}_i)c(w_j)\\
            &=-\frac{1}{2}\sum_k\dom{c}(\dom{L}w_k)c(Lw_k).
        \end{aligned}
    \end{equation*}
    Here, \(\dom{L}w_k= \sum_i \langle F_*\dom{w}_i,Lw_k\rangle \dom{w}_i\) by definition.
    Continuing the calculation
    \begin{equation*}
        \begin{aligned}
            -\frac{1}{2}\sum_k\dom{c}(\dom{L}w_k)c(Lw_k)=\frac{1}{4}&\sum_k\Big(\alpha^{-2}\dom{c}(\dom{L}w_k)^2+\alpha^2c(Lw_k)^2 \\
            &\qquad -\big(\alpha^{-1}\dom{c}(\dom{L}w_k)+\alpha c(Lw_k)\big)^2\Big)
        \end{aligned}
    \end{equation*}
    where \(\alpha\) is some scalar function to be specified.
    Since \(\alpha\dom{c}(\dom{L}w_k)+\alpha^{-1}c(Lw_k)\) is skew-symmetric, its square is nonpositive, so the negative of its square will be nonnegative.
    This implies
    \begin{equation*}
        -\frac{1}{2}\sum_k\dom{c}(\dom{L}w_k)c(Lw_k)\geq\frac{1}{4}\sum_{k}\left(\alpha^{-2}\dom{c}(\dom{L}w_k)^2+\alpha^2c(Lw_k)^2\right).
    \end{equation*}
    Thus, it remains to investigate \(\sum_k\alpha^2\dom{c}(\dom{L}w_k)^2\) and \(\sum_k\alpha^{-2}c(Lw_k)^2\).
    Starting with the latter, we compute
    \begin{equation*}
        \begin{aligned}
            \frac{1}{2}\sum_k\alpha^2c(Lw_k)^2&=\frac{1}{2}\alpha^2\sum_k\sum_{i,j}\langle Lw_k,w_i\rangle\langle Lw_k,w_j \rangle c(w_i)c(w_j)\\
            &=\frac{1}{2}\alpha^2\sum_{i,j}\langle Lw_i,Lw_j\rangle c(w_i)c(w_j)\\
            &=\frac{1}{2}\alpha^2\sum_{i,j}\langle \mathcal{R}w_i,w_j\rangle c(w_i)c(w_j)\\
            &=-\alpha^2\frac{F^*\Sc}{4}
        \end{aligned}
    \end{equation*}
    where the last equality is by the Lichnerowicz calculation.

    To analyze the former, start by choosing orthonormal bases \(\{\dom{e}_i\}\) on \(\dom{M}\) and \(\{e_j\}\) on \(M\) for which \(F_*\dom{e}_i=\mu_ie_i\).
    Then,
    \begin{equation*}
        \begin{aligned}
            \frac{1}{2}\sum_k\alpha^{-2}\dom{c}(\dom{L}w_k)^2&=\frac{1}{2}\alpha^{-2}\sum_k\sum_{i,j}\langle F_*\dom{w}_i,Lw_k\rangle\langle F_*\dom{w}_j,Lw_k\rangle\dom{c}(\dom{w}_i)\dom{c}(\dom{w}_j)\\
            &=\frac{1}{2}\alpha^{-2}\sum_{i,j}\langle LF_*\dom{w}_i,LF_*\dom{w}_j\rangle\dom{c}(\dom{w}_i)\dom{c}(\dom{w}_j)\\
            &=-\frac{1}{8}\alpha^{-2}\sum_{i,j,k,l}\mu_i\mu_k\mu_j\mu_l\langle R(e_i,e_k)e_j,e_l\rangle\dom{c}(\dom{e}_i)\dom{c}(\dom{e}_k)\dom{c}(\dom{e}_j)\dom{c}(\dom{e}_l)\\
            &=-\frac{1}{4}\alpha^{-2}\sum_{i,j}\mu_i^2\mu_j^2\langle R(e_i,e_j)e_j,e_i\rangle\\
            &\geq-\alpha^{-2}\onorm{\wedge^2dF}^2\frac{F^*\Sc}{4},
        \end{aligned}
    \end{equation*}
    where the last equality is, again, by the same calculation that occurs in the Lichnerowicz formula.
    Taking \(\alpha=\sqrt{\onorm{\wedge^2 dF}}\), we conclude
    \begin{equation*}
        \begin{aligned}
            -\frac{1}{2}\sum_{i,j}\left<\mathcal{R}F_*\dom{w}_i,w_j\right>\dom{c}(\dom{w}_i)c(w_j)&\geq-\alpha^{-2}\onorm{\wedge^2 dF}^2\frac{F^*\Sc}{8}-\alpha^2\frac{F^*\Sc}{8} \\
            &=-\onorm{\wedge^2 dF}\frac{F^*\Sc}{4}.
        \end{aligned}
    \end{equation*}
    To finish the proof
    \begin{equation*}
        \int_{\dom{M}}\pnorm{D\sigma}^2=\int_{\dom{M}}\langle D^2\sigma,\sigma\rangle\geq\int_{\dom{M}}\pnorm{\nabla \sigma}^2+\int_{\dom{M}}\frac{\dom{\Sc}}{4}\pnorm{\sigma}^2-\int_{\dom{M}}\onorm{\wedge^2 dF}\frac{F^*\Sc}{4}\pnorm{\sigma}^2.
    \end{equation*}
\end{proof}

A more precise estimate can be obtained by carefully analyzing the boundary terms that appear when applying Stokes's theorem.
In doing so, the mean curvatures of both the domain and target manifolds will appear.
Before stating the result, we establish some notation.
In Section \ref{subsec:twisted dirac operator on cone}, the connection \(\conn[\partial,\partial]\) was defined by removing the contributions from the radial components.
In this section, we define the connection \(\conn[\partial, \partial]\) on \(S(T\dom{M} \oplus F^*TM)|_{\bnd{\dom{M}}}\) by removing the contributions from the unit inner normal components \(\dom{\nu}\) and \(\nu\) of \(\bnd{\dom{M}}\) and \(\bnd{M}\), respectively.

To be precise, define tensors \(T^{\partial,0}\) and \(T^{0,\partial}\) as follows:
For orthonormal frames \(\{\dom{e}_i\}\) and \(\{e_i\}\) of \(\bnd{\dom{M}}\) and \(\bnd{M}\), let
\begin{align*}
    T^{\partial, 0}_X \sigma &= \frac{1}{2} \sum_i \dom{\omega}_{\dom{\nu}}^i(X) \cbar{\dom{\nu}} \cbar{\dom{e}_i} \sigma\\
    T^{0, \partial}_X \sigma &= \frac{1}{2} \sum_i (F^*\omega_{\nu}^i)(X) c(\nu) c(e_i) \sigma
\end{align*}
where \(X \in T\bnd{\dom{M}}\), \(\sigma\) is a section of \(S(T\dom{M} \oplus F^*TM)|_{\bnd{\dom{M}}}\), and
\begin{align*}
    \dom{\omega}_{\dom{\nu}}^i(X) &\coloneqq \inner{\conn[]_X \dom{\nu}, \dom{e}_i} \\
    (F^*\omega_{\nu}^i)(X) &\coloneqq \inner{\conn[]_{F_*X} \nu, e_i}.
\end{align*}
It is not hard to see that \(T^{\partial, 0}\) and \(T^{0, \partial}\) are independent of the frames \(\{\dom{e}_i\}\) and \(\{e_i\}\).

Now, for \(X\) a tangent vector of \(\bnd{\dom{M}}\), the connection \(\conn[\partial, \partial]\) is defined by
\begin{equation*}
    \conn[\partial, \partial]_X \sigma = \conn[]_X \sigma - T^{\partial, 0}_X \sigma - T^{0,\partial}_X \sigma.
\end{equation*}
One can manually check that
\begin{align*}
    \conn[\partial, \partial] \cbar{\dom{\nu}} \sigma &= \cbar{\dom{\nu}} \conn[\partial, \partial] \sigma, \\
    \conn[\partial, \partial] c(\nu) \sigma &= c(\nu) \conn[\partial, \partial] \sigma.
\end{align*}

Using the above notations, we extend Theorem \ref{thm:scalar curvature estimate}.
\begin{proposition}\label{prop:scalar-mean curvature estimate}
    Let \((\dom{M}, \bnd{\dom{M}}, \dom{g})\) and \((M, \bnd{M}, g)\) be manifolds with FACS and smooth boundaries \(\bnd{\dom{M}}\), \(\bnd{M}\).
    Let \(F\colon \dom{M} \to M\) be an asymptotically conical map and let \(F_\partial\) denote its restriction to the boundary.
    Suppose the curvature operator \(\mathcal{R}\) on M is nonnegative and suppose the second fundamental form \(\mathcal{A}\) along the boundary of \(\bnd{M}\) is also nonnegative.
    If \(\sigma\) is a smooth section of \(S(T\dom{M} \oplus F^*TM)\) vanishing near the singular stratum of \(\dom{M}\), then the following inequality holds:
    \begin{multline}
        \int_{\dom{M}}\pnorm{D\sigma}^2\geq \int_{\dom{M}}\pnorm{\nabla \sigma}^2+\int_{\dom{M}} \frac{1}{4}(\dom{\Sc}-\onorm{\wedge^2dF}F^*\Sc )\pnorm{\sigma}^2\\
        + \int_{\bnd{\dom{M}}} \inner{D^{\partial, \partial} \sigma, \sigma} + \int_{\bnd{\dom{M}}} \frac{1}{2} (\dom{H} - \onorm{d(F_\partial)} (F_\partial)^*H )\pnorm{\sigma}^2.
    \end{multline}
    Here, \(D\) is the twisted Dirac operator on \(S(T\dom{M} \oplus F^*TM)\)  and \(D^{\partial, \partial}\) is the boundary Dirac operator defined as
    \begin{equation*}
        D^{\partial, \partial} \sigma = \sum_i \cbar{\dom{\nu}}\cbar{\dom{e}_i} \conn[\partial, \partial] \sigma.
    \end{equation*}
    In the above, \(\dom{\nu}\) is the inner unit normal and the vectors \(\{\dom{e}_i\}\) form a local orthonormal frame of the boundary of \(\dom{M}\).
\end{proposition}
\begin{proof}
    Observe that
    \begin{equation*}
        \int_{\dom{M}} \pnorm{D \sigma}^2 = \int_{\dom{M}} \inner{D^2 \sigma, \sigma} + \int_{\bnd{\dom{M}}} \inner{\cbar{\dom{\nu}} D \sigma, \sigma}.
    \end{equation*}

    Rewriting the last term in the above
    \begin{align*}
        \int_{\bnd{\dom{M}}} \inner{\cbar{\dom{\nu}} D \sigma, \sigma} &= -\int_{\bnd{\dom{M}}} \inner{\conn[]_{\dom{\nu}} \sigma, \sigma} + \int_{\bnd{\dom{M}}} \inner{\sum_i \cbar{\dom{\nu}} \cbar{\dom{e}_i} \conn[]_{\dom{e}_i} \sigma, \sigma}\\
        &= -\int_{\bnd{\dom{M}}} \inner{\conn[]_{\dom{\nu}} \sigma, \sigma} + \int_{\bnd{\dom{M}}} \inner{\sum_i \cbar{\dom{\nu}} \cbar{\dom{e}_i} \conn[\partial, \partial] \sigma, \sigma}\\
        &\qquad\qquad + \int_{\bnd{\dom{M}}} \inner{\frac{1}{2} \sum_{i, j} \inner{\conn[]_{\dom{e}_i} \dom{\nu}, \dom{e}_j} \cbar{\dom{\nu}} \cbar{\dom{e}_i} \cbar{\dom{\nu}} \cbar{\dom{e}_j} \sigma, \sigma}\\
        &\qquad\qquad + \int_{\bnd{\dom{M}}} \inner{\frac{1}{2} \sum_{i, j} \inner{\conn[]_{F_*\dom{e}_i} \nu, e_j} \cbar{\dom{\nu}} \cbar{\dom{e}_i} c(\nu) c(e_j) \sigma, \sigma}.
    \end{align*}
    To simplify the notation, define \(\dom{c}_\partial(\dom{e}_i) \coloneqq \dom{c}(\dom{\nu}) \dom{c}(\dom{e}_i)\) and \(c_\partial(e_i) \coloneqq c(\nu) c(e_i)\).
    Rewriting the above,
    \begin{equation}\label{eqn:boundary dirac operator term}
        \begin{split}
            \int_{\bnd{\dom{M}}} \inner{\cbar{\dom{\nu}} D \sigma, \sigma} &= -\int_{\bnd{\dom{M}}} \inner{\conn[]_{\dom{\nu}} \sigma, \sigma} + \sum_i \int_{\bnd{\dom{M}}} \inner{D^{\partial, \partial} \sigma, \sigma}\\
            &\qquad\qquad + \int_{\bnd{\dom{M}}} \inner{\underbrace{\frac{1}{2} \sum_{i, j} \inner{\conn[]_{\dom{e}_i} \dom{\nu}, \dom{e}_j} \dom{c}_\partial(\dom{e}_i) \dom{c}_\partial(\dom{e}_j)}_{\text{Term 1}}\sigma , \sigma}\\
            &\qquad\qquad + \int_{\bnd{\dom{M}}} \inner{\underbrace{\frac{1}{2} \sum_{i, j} \inner{\conn[]_{F_*\dom{e}_i} \nu, e_j} \dom{c}_\partial(\dom{e}_i) c_\partial(e_j)}_{\text{Term 2}}\sigma , \sigma}.
        \end{split}
    \end{equation}

    Isolating Term 1, we compute
    \begin{equation*}
        \frac{1}{2} \sum_{i, j} \inner{\conn[]_{\dom{e}_i} \dom{\nu}, \dom{e}_j} \dom{c}_\partial(\dom{e}_i) \dom{c}_\partial(\dom{e}_j) = \frac{1}{2} \sum_i \inner{\dom{\nu}, \conn[]_{\dom{e}_i} \dom{e}_i} - \underbrace{\frac{1}{2} \sum_{i \neq j} \inner{\dom{\nu}, \conn[]_{\dom{e}_i} \dom{e}_j} \cbar{\dom{e}_i} \cbar{\dom{e}_j}}_{=\, 0}
    \end{equation*}
    where the last term is 0 since \(\inner{\dom{\nu}, \conn[]_{\dom{e}_i} \dom{e}_j}\) is symmetric with respect to \(\dom{e}_i\) and \(\dom{e}_j\) while \(\cbar{\dom{e}_i} \cbar{\dom{e}_j}\) is skew-symmetric with respect to \(\dom{e}_i\) and \(\dom{e}_j\) when \(i \neq j\).
    Hence, Term 1 is just the mean curvature of \(\bnd{\dom{M}}\).

    Next, we analyze Term 2.
    Start by taking frames \(\{\dom{e}_i\}\) and \(\{e_i\}\) for which \(F_*\dom{e}_i = \mu_i e_i\).
    Observe that \(\inner{\conn[]_{F_*\dom{e}_i} \nu, e_j} = -\inner{\mathcal{A} F_*\dom{e}_i, e_j}\).
    If \(B\) denotes the square root of \(\mathcal{A}\), then
    \begin{align*}
        -\inner{\mathcal{A} F_*\dom{e}_i, e_j} &= -\inner{B F_*\dom{e}_i, B e_j} = -\sum_k \inner{BF_*\dom{e}_i, e_k} \inner{Be_j, e_k}
    \end{align*}

    Hence, Term 2 can be rewritten as
    \begin{align*}
        \frac{1}{2} \sum_{i, j} \inner{\conn[]_{F_*\dom{e}_i} \nu, e_j} \dom{c}_\partial(\dom{e}_i) c_\partial(e_j) &= -\frac{1}{2}\sum_{i,j,k} \inner{BF_*\dom{e}_i, e_k} \inner{Be_j,e_k} \dom{c}_\partial(\dom{e}_i) c_\partial(e_j) \\
        &= -\frac{1}{2} \sum_k \dom{c}_\partial(\dom{B} e_k) c_\partial(B e_k) \\
        &= \frac{1}{4} \sum_k \Big(\alpha^{-2} \dom{c}_\partial(\dom{B}e_k)^2 + \alpha^2 c_\partial(Be_k)^2 \\
        &\qquad\qquad\qquad -\big(\alpha^{-1} \dom{c}_\partial(\dom{B}e_k) + \alpha c_\partial(Be_k)\big)^2\Big)
    \end{align*}
    where \(\dom{B}e_k \coloneqq \sum_i \left< F_*\dom{e}_i, Be_k \right> \dom{e}_i\) and \(\alpha\) is some scalar function to be specified.
    Notice that \(\alpha^{-1} \dom{c}_\partial(\dom{B} e_k) + \alpha c_\partial(B e_k)\) is skew-symmetric and hence the negative of its square is nonnegative.
    At the same time,
    \begin{align*}
        \frac{1}{4} \alpha^2 \sum_k c_\partial(Be_k)^2 = -\frac{1}{4} \alpha^2 \sum_k \pnorm{B e_k}^2 &= -\frac{1}{4} \alpha^2 \sum_k \left<\mathcal{A}e_k, e_k\right> \\
        &= -\frac{1}{4} \alpha^2 F^*H.
    \end{align*}
    and also
    \begin{align*}
        \frac{1}{4} \alpha^{-2} \sum_k \dom{c}_\partial(\dom{B} e_k)^2 &=  \frac{1}{4} \alpha^{-2} \sum_{k,i,j} \inner{F_*\dom{e}_i, Be_k} \inner{F_*\dom{e}_j, Be_k} \dom{c}_\partial(\dom{e}_i) \dom{c}_\partial(\dom{e}_j)\\
        &= -\frac{1}{4} \alpha^{-2} \sum_{k,i} \left<F_*\dom{e}_i, Be_k\right>^2 \\
        &= -\frac{1}{4} \alpha^{-2} \sum_{k,i} \left<\mu_i Be_i, e_k\right>^2 \\
        &\geq -\frac{1}{4} \alpha^{-2} \sum_i \onorm{dF}^2 \pnorm{B e_i}^2 \\
        &= -\frac{1}{4} \alpha^{-2} \onorm{dF}^2 F^*H.
    \end{align*}

    We have established the following estimate for Term 2:
    \begin{align*}
        \frac{1}{2} \sum_{i, j} \inner{\conn[]_{F_*\dom{e}_i} \nu, e_j} \dom{c}_\partial(\dom{e}_i) c_\partial(e_j) &\geq -\frac{1}{4} \alpha^{-2} \onorm{dF}^2 F^*H - \frac{1}{4} \alpha^2 F^*H \\
        &= -\frac{\onorm{dF}}{2} F^*H
    \end{align*}
    by taking \(\alpha = \sqrt{\onorm{dF}}\).

    To summarize,
    \begin{multline}\label{eqn:final estimate of dirac boundary term}
        \int_{\bnd{\dom{M}}} \inner{\cbar{\dom{\nu}} D \sigma, \sigma} \geq -\int_{\bnd{\dom{M}}} \inner{\conn[]_{\dom{\nu}} \sigma, \sigma} + \int_{\bnd{\dom{M}}} \inner{D^{\partial, \partial}\sigma, \sigma}\\
        + \int_{\bnd{\dom{M}}} \frac{\dom{H}}{2} \pnorm{\sigma}^2 - \int_{\bnd{\dom{M}}} \onorm{dF} \frac{F^*H}{2} \pnorm{\sigma}^2.
    \end{multline}

    At the same time, by the proof of Theorem \ref{thm:scalar curvature estimate},
    \begin{equation}\label{eqn:final estimate of square of dirac operator term}
        \int_{\dom{M}} \inner{D^2 \sigma, \sigma} = \int_{\dom{M}} \inner{\nabla^*\nabla \sigma, \sigma} + \int_{\dom{M}} \frac{\dom{\Sc}}{4} \pnorm{\sigma}^2 - \int_{\dom{M}} \frac{F^*\Sc}{4} \pnorm{\sigma}^2.
    \end{equation}.

    Integrating by parts,
    \begin{equation}\label{eqn:integration by parts of laplacian}
        \int_{\dom{M}} \inner{\nabla^*\nabla \sigma, \sigma} = \int_{\dom{M}} \pnorm{\conn[] \sigma}^2 + \int_{\bnd{\dom{M}}} \inner{\conn[]_{\dom{\nu}} \sigma, \sigma}.
    \end{equation}

    Combining Equations \eqref{eqn:final estimate of dirac boundary term}, \eqref{eqn:final estimate of square of dirac operator term}, and \eqref{eqn:integration by parts of laplacian} yields the estimate.
\end{proof}

The conclusion of Proposition \ref{prop:scalar-mean curvature estimate} can be refined if we restrict to smooth sections \(\sigma\) vanishing near the singular stratum and satisfying a particular, pointwise-defined, boundary condition \(B\).
\begin{definition}
    A smooth section \(\sigma\) of \(S(T\dom{M} \oplus F^*TM)\) satisfies the boundary condition \(B\) if
    \begin{equation*}
        \sqrt{-1}\cbar{\dom{\nu}} c(\nu) \sigma|_{\bnd{\dom{M}}} = -\sigma|_{\bnd{\dom{M}}}
    \end{equation*}
    where \(\dom{\nu}\) and \(\nu\) are the unit inner normal vectors of \(\bnd{\dom{M}}\) and \(\bnd{M}\).\footnote{Recall from Section \ref{subsec:twisted dirac operator on cone} that we view \(S(T\dom{M} \oplus F^*TM)\) as a \(\Cl(T\dom{M} \oplus F^*TM)\)-module, which means that \(\dom{c}(\dom{\nu})\) and \(c(\nu)\) anticommute.}

    Such a section \(\sigma\) is said to satisfy the absolute boundary condition \(B\).
    For brevity, we sometimes say \(\sigma\) is absolute.
\end{definition}

There is also a naturally associated complementary boundary condition \(B^{\perp}\).
\begin{definition}
    A smooth section \(\sigma\) of \(S(T\dom{M} \oplus F^*TM)\) satisfies the boundary condition \(B^{\perp}\) if
    \begin{equation*}
        \sqrt{-1}\cbar{\dom{\nu}} c(\nu) \sigma|_{\bnd{\dom{M}}} = \sigma|_{\bnd{\dom{M}}}
    \end{equation*}
    where again \(\dom{\nu}\) and \(\nu\) are the unit inner normal vectors of \(\bnd{\dom{M}}\) and \(\bnd{M}\).

    Such a section \(\sigma\) is said to satisfy the relative boundary condition \(B^\perp\).
    For brevity, we sometimes say \(\sigma\) is relative.
\end{definition}
\begin{remark}
    If \(\sigma\) is absolute and \(\tilde{\sigma}\) is relative, then \(\sigma|_{\bnd{\dom{M}}}\) is orthogonal to \(\tilde{\sigma}|_{\bnd{\dom{M}}}\) pointwise.
    This is true because \(\sigma|_{\bnd{\dom{M}}}\) is a -1-eigenvalue of \(\sqrt{-1}\cbar{\dom{\nu}} c(\nu)\) and \(\tilde{\sigma}|_{\bnd{\dom{M}}}\) is a +1-eigenvalue of \(\sqrt{-1}\cbar{\dom{\nu}} c(\nu)\).
\end{remark}
\begin{remark}\label{rem:boundary conditions in case of clifford bundle}
    Suppose that \((\dom{M}, \bnd{\dom{M}}, \dom{g}) = (M, \bnd{M}, g)\), \(M\) is even-dimensional and the map \(F\) is just the identity map.
    Then, \(S(TM \oplus TM)\) can be canonically identified with \(\Cl(TM)\) such that the actions \(\dom{c}(w)\) and \(c(w)\) are given by
    \begin{align*}
        \dom{c}(w) v_1 v_2 \dotsm v_p &= w v_1 v_2 \dotsm v_p \\
        c(w) v_1 v_2 \dotsm v_p &= (-1)^{p+1} \sqrt{-1}v_1 v_2 \dotsm v_p w.
    \end{align*}
    Every element of \(\Cl(TM)|_{\bnd{M}}\) can be written locally as a sum of products of the vectors \(\nu\) and \(\{e_i\}\), where \(\nu\) is the unit inner normal vector and \(\{e_i\}\) is any local frame for \(T\bnd{M}\).
    In particular, a section \(\sigma\) of \(\Cl(TM)\) restricted to the boundary can be locally written as a sum of products of the vectors \(\nu\) and \(\{e_i\}\).
    \(\sigma\) is absolute if and only if \(\nu\) does not appear in any term of this sum.
    \(\sigma\) is relative if and only if \(\nu\) appears in every term of this sum.
\end{remark}

The following statements concerning the boundary conditions \(B\) and \(B^{\perp}\) are true:
\begin{lemma}
    Suppose \(\sigma\) is absolute (relative).
    Then,
    \begin{enumerate}
        \item \(\cbar{X} \sigma\) and \(c(Y) \sigma\) are absolute (relative)  for \(X \in T\bnd{\dom{M}}\) and \(Y \in T\bnd{M}\).
        \item The derivative \(\conn[\partial, \partial] \sigma\) is absolute (relative).
        \item \(\cbar{\dom{\nu}} \sigma\) and \(c(\nu) \sigma\) are relative (absolute). Moreover,
        \begin{equation*}
            c(\nu) \sigma = \sqrt{-1} \cbar{\dom{\nu}} \varepsilon\sigma
        \end{equation*}
        where \(\varepsilon = -1\) if \(\sigma\) is absolute and \(\varepsilon = +1\) if \(\sigma\) is relative.
    \end{enumerate}
\end{lemma}
Using these boundary conditions, Proposition \ref{prop:scalar-mean curvature estimate} can be refined as follows:
\begin{corollary}\label{cor:scalar-mean curvature estimate with boundary condition}
    Suppose \(\sigma\) satisfies the hypotheses of Proposition \ref{prop:scalar-mean curvature estimate}.
    If \(\sigma\) is absolute or relative, then
    \begin{multline}\label{eqn:scalar-mean curvature estimate with boundary condition}
        \int_{\dom{M}}\pnorm{D\sigma}^2\geq \int_{\dom{M}}\pnorm{\nabla \sigma}^2+\int_{\dom{M}} \frac{1}{4}(\dom{\Sc}-\onorm{\wedge^2dF}F^*\Sc )\pnorm{\sigma}^2\\
        + \int_{\bnd{\dom{M}}} \frac{1}{2} (\dom{H} - \onorm{d(F_\partial)} (F_\partial)^*H )\pnorm{\sigma}^2.
    \end{multline}
\end{corollary}
\begin{proof}
    The inner product \(\inner{D^{\partial, \partial} \sigma, \sigma}\) vanishes since \(D^{\partial, \partial} \sigma\), which is defined by \(D^{\partial, \partial} \sigma = \sum_i \cbar{\dom{\nu}} \cbar{\dom{e}_i} \conn[\partial, \partial] \sigma\), is relative (absolute) if \(\sigma\) is absolute (relative).
\end{proof}

\section{Fredholmness of twisted Dirac operators}\label{sec:fredholmness}
Throughout this section, \((\dom{M}, \bnd{\dom{M}},\dom{g})\) and \((M, \bnd{\dom{M}},g)\) will be spin manifolds with FACS and \(F \colon \dom{M}\to M\) an asymptotically conical map.
Before proving self-adjointness, we present some functional analytic preliminaries.

\subsection{Functional analytic preliminaries}\label{subsec:functional analytic preliminaries}
\begin{definition}\label{def:spaces of functions} \mbox{}
    Let \(E\) denote a Hermitian vector bundle over \(M\) with Hermitian connection and \(V \coloneqq T\dom{M} \oplus F^*TM\).
    \begin{enumerate}
        \item \(C_c^\infty(M, E)\) is defined to be the set of smooth compactly supported sections of \(E\) vanishing near the singular stratum.
        \item \(H^1(M, E)\) is defined to be the completion of \(C_c^\infty(M, E)\) with respect to the norm
        \begin{equation*}
            \norm{\varphi}_{H^1}\coloneqq (\norm{\varphi}_{L^2}^2+\norm{\nabla\varphi}_{L^2}^2)^{1/2}
        \end{equation*}
        \item \(C_c^{\infty}(\dom{M}, SV; B)\) denotes the set of smooth compactly supported sections of \(SV\) vanishing near the singular stratum and satisfying the absolute boundary condition \(B\).
        \item \(H^1(\dom{M}, SV; B)\) is the completion of \(C_c^\infty(\dom{M}, SV; B)\) with respect to the \(H^1\)-norm.
    \end{enumerate}
\end{definition}
The twisted Dirac operator \(D\) will be regarded as an unbounded operator of \(L^2(\dom{M}, SV)\) with domain \(C_c^{\infty}(\dom{M}, SV)\).
The twisted Dirac operator \(D_B\) with absolute boundary condition \(B\) will be regarded as an unbounded operator of \(L^2(\dom{M}, SV)\) with domain \(C_c^{\infty}(\dom{M}, SV; B)\).
Observe that \(D_B\) is symmetric.

The minimal extension is denoted \(D_{min}\) and is the closure of \(D\) with respect to the graph norm.
The maximal extension \(D_{max}\) is the functional analytic adjoint \(D^*\). The same definitions are made for \((D_B)_{min}\) and \((D_B)_{max}\). Clearly, \(H^1(\dom{M}, SV; B) \subseteq \operdom(D_B)_{min}\).
Our geometric hypotheses combined with Corollary \ref{cor:scalar-mean curvature estimate with boundary condition} show that if \(\sigma \in C_c^{\infty}(\dom{M}, SV; B)\), then \(\norm{\nabla \sigma}_{L^2}^2 \leq \norm{D \sigma}_{L^2}^2\).\footnote{Actually, one can deduce \(\norm{\nabla \sigma}^2 \leq C\norm{D \sigma}^2\) without the geometric assumptions by applying Remark \ref{rem:multiplication with scalar curvature is a bounded map}.}
This implies \(\operdom(D_B)_{min} \subseteq H^1(\dom{M}, SV; B)\).

The next lemma will be useful in Section \ref{subsec:self-adjointness in general case - facs manifolds} when showing that the twisted Dirac operator in the general case differs from that of the model case by error terms which can be made negligible.

\begin{proposition}\label{prop:multiplication by 1/r is bounded}
    Let \(E\) be a Hermitian vector bundle with Hermitian connection over a compact manifold \(M\) with FACS.
    Let \(r(p) \coloneqq dist(p, \cat{C})\) denote the distance from \(p\) to the singular stratum.
    If the links of \(M\) have dimension at least two, then multiplication by \(1/r\) defines a bounded linear map from \(H^1(M, E)\) to \(L^2(M, E)\).
    Moreover, the operator norm is bounded above by a constant that depends only on the geometry and topology of \(M\).
\end{proposition}
\begin{proof}
    It is clear that multiplication by \(1/r\) is a bounded operator when restricted to \(H^1\)-sections supported uniformly away from the singularities.
    Thus, by localizing at the singularities, to prove the lemma, it is sufficient to consider the case where \(M = \R^\hdim \times \cone{L}\) and restrict to smooth compactly supported sections vanishing near the singular stratum.
    In addition, the proof of the asymptotically conical case follows from the conical case by accounting for error terms, so we only consider the case where \(g = \pi^*g_B \oplus dr^2 \oplus g_L(b)\).

    Let \(\ldim\) denote the dimension of \(L\).
    Take a smooth compactly supported section \(\sigma\) vanishing near the singular stratum.
    Applying integration by parts, we compute\footnote{Throughout this proof, if the domain of integration is not specified, it is assumed to be \(\R^\hdim \times L\).}
    \begin{align*}
        \norm{\frac{1}{r} \sigma}_{L^2}^2 &= \iint_{0}^{\infty} \pnorm{\sigma}^2 r^{\ldim-2} dr \\
        &= -\frac{2}{\ldim-1} \iint_{0}^{\infty} \inner{\conn[]_{\partial_r} \sigma, \sigma} r^{\ldim-1}dr \\
        &\leq \frac{2}{\ldim-1} \norm{\conn[]_{\partial_r} \sigma}_{L^2} \norm{\frac{1}{r} \sigma}_{L^2}.
    \end{align*}
    Dividing both sides by \(\norm{\frac{1}{r} \sigma}_{L^2}\), we get
    \begin{equation*}
        \norm{\frac{1}{r} \sigma}_{L^2} \leq \frac{2}{\ldim-1} \norm{\conn[]_{\partial_r} \sigma}_{L^2} \leq \frac{2}{\ldim-1} \norm{\sigma}_{H^1}
    \end{equation*}
\end{proof}
\begin{remark}\label{rem:multiplication with scalar curvature is a bounded map}
    Combining Proposition \ref{prop:multiplication by 1/r is bounded} with Corollary \ref{cor:scalar curvature of cone} shows that if the links have dimension at least two, then multiplication by \(\sqrt{\Sc}\) is a bounded map from \(H^1\) to \(L^2\).
\end{remark}

Next, we prove a Rellich-Kondrachov compactness theorem for manifolds with FACS.
\begin{proposition}\label{prop:compactness theorem for facs manifolds}
    Let \(E\) be a Hermitian vector bundle with Hermitian connection over a compact manifold \(M\) with FACS.
    If the links of \(M\) have dimension at least two, then the inclusion of \(H^1(M, E)\) into \(L^2(M, E)\) is compact.
\end{proposition}
\begin{proof}
    The idea of the proof is to start by taking an exhaustion of \(M\) by smooth compact manifolds, with boundary, approaching the singular stratum.
    Then, apply the usual compactness theorem to each compact manifold in the exhaustion.
    What needs to be shown is that a sequence \(\{\sigma_k\}\) uniformly bounded in the \(H^1\)-norm becomes uniformly small when restricted to the complement of such an exhaustion.
    For this, it is sufficient to establish the statement in a conical neighborhood of any singular point.
    By accounting for error terms, it is enough to consider the model case where \(M = \R^\hdim \times \cone{L}\) has the conical metric \(g = \pi^*g_B \oplus dr^2 \oplus g_L(b)\).
    Here, \(L\) is an \(\ldim\)-dimensional smooth compact manifold without boundary.

    Take a sequence \(\{\sigma_k\}\) of smooth compactly supported sections supported away from the singular stratum.
    Without loss of generality, assume that the \(H^1\)-norms are bounded above by 1.

    Given some \(\varepsilon\), applying integration by parts shows\footnote{Throughout this proof, if the domain of integration is not specified, it is assumed to be \(\R^\hdim \times L.\)}
    \begin{equation*}
        \iint_{0}^{\varepsilon} \pnorm{\sigma_k}^2 r^\ldim dr = -\frac{2}{\ldim+1} \iint_{0}^{\varepsilon} \inner{\conn[]_{\partial_r} \sigma_k, \sigma_k} r^{\ldim+1} dr + \frac{1}{\ldim+1} \int \pnorm{\sigma_k|_{r=\varepsilon}}^2 \varepsilon^{\ldim+1}.
    \end{equation*}
    Estimating the first term on the right-hand side,
    \begin{align*}
        -\frac{2}{\ldim+1} \iint_{0}^{\varepsilon} \inner{\conn[]_{\partial_r} \sigma_k, \sigma_k} r^{\ldim+1} dr &\leq \frac{2}{\ldim+1} \norm{\conn[] \sigma_k}_{L^2} \iint_0^\varepsilon \pnorm{r\sigma_k}^2 r^l dr  \\
        &\leq \frac{2 \varepsilon^2}{\ldim+1}.
    \end{align*}
    Estimating the boundary term,
    \begin{align*}
        \frac{1}{\ldim+1} \int \pnorm{\sigma_k|_{r=\varepsilon}}^2 \varepsilon^{\ldim+1} &= \frac{\varepsilon}{\ldim+1} \iint_{0}^{\varepsilon} \frac{d}{dr} \big(\pnorm{\sigma_k}^2 r^\ldim\big) dr \\
        &= \frac{2\varepsilon}{\ldim+1} \iint_{0}^{\varepsilon} \inner{\conn[]_{\partial_r} \sigma_k, \sigma_k} r^\ldim dr \\
        &\qquad\qquad + \frac{l\varepsilon}{\ldim+1} \iint_{0}^{\varepsilon} \pnorm{\sigma_k}^2 r^{\ldim-1} dr.
    \end{align*}
    By the Cauchy-Schwarz inequality,
    \begin{equation*}
        \frac{2\varepsilon}{\ldim+1} \iint_{0}^{\varepsilon} \inner{\conn[]_{\partial_r} \sigma_k, \sigma_k} r^\ldim dr \leq \frac{2\varepsilon}{\ldim+1} \norm{\sigma_k}_{L^2} \norm{\conn[] \sigma_k}_{L^2} \leq \frac{2\varepsilon}{\ldim+1}.
    \end{equation*}
    By Proposition \ref{prop:multiplication by 1/r is bounded},
    \begin{equation*}
        \frac{l\varepsilon}{\ldim+1} \iint_{0}^{\varepsilon} \pnorm{\sigma_k}^2 r^{\ldim-1} dr \leq \frac{l\varepsilon}{\ldim+1} \norm{\frac{1}{r^{1/2}} \sigma_k}_{L^2}^2 \leq \frac{Cl\varepsilon}{\ldim+1}.
    \end{equation*}

    To summarize, the above estimates show that
    \begin{equation}\label{eqn:L^1 compactness estimate}
        \iint_{0}^{\varepsilon} \pnorm{\sigma_k}^2 r^\ldim dr \leq C \varepsilon.
    \end{equation}
    for some constant \(C\), that is, the \(L^2\)-norms of the \(\{\sigma_k\}\) restricted to \(\R^\hdim \times (0, \varepsilon) \times L\) are bounded uniformly by a constant multiple of \(\varepsilon\).

    Let \(K_\rho\) denote the set of points that are at least distance \(\rho\) from the singular stratum.
    Take an exhaustion \(\{K_{\rho_k}\}\), where \(\rho_k \to 0\) as \(k \to \infty\).
    The Rellich-Kondrachov compactness theorem holds for each \(K_{\rho_k}\).
    Restricting \(\{\sigma_i\}\) to \(K_{\rho_1}\), there is a subsequence \(\sigma_i^{(1)}\) that is Cauchy on \(K_{\rho_1}\).
    Taking this subsequence and restricting it to \(K_{\rho_2}\), there is a subsequence of \(\{\sigma_i^{(1)}\}\), denoted \(\{\sigma_i^{(2)}\}\), which is Cauchy on \(K_{\rho_2}\).
    Continuing this process inductively, we can construct \(\sigma_i^{(n)}\) which is Cauchy on \(K_{\rho_n}\) and is a subsequence of \(\{\sigma_i^{(n-1)}\}\).
    Consider the sequence \(\{\sigma_i^{(i)}\}\).
    This sequence is Cauchy on each \(K_{\rho_n}\) by construction.
    To show that it is Cauchy on \(M\), given \(\varepsilon > 0\), take \(k_0\) large enough so that \(\rho_{k_0} < \varepsilon\).
    By Equation \eqref{eqn:L^1 compactness estimate}, the sequence \(\{\sigma_i^{(i)}\}\) restricted to \(M \setminus K_{\rho_{k_0}}\) is uniformly bounded by \(C \varepsilon\).
    Take \(i_0\) large enough so that if \(i,j \geq i_0\), then \(\norm{(\sigma_i^{(i)}|_{ K_{\rho_{k_0}} }) - (\sigma_j^{(j)}|_{ K_{\rho_{k_0}} })}_{L^2} < \varepsilon\).
    Thus, for \(i,j \geq i_0\),
    \begin{equation*}
        \norm{ \sigma_i^{(i)} - \sigma_j^{(j)} }_{L^2} < C \varepsilon + \varepsilon.
    \end{equation*}
\end{proof}
\begin{remark}\label{rem:compactness theorem for C_1-manifolds with C_1-boundary}
    %Definition \ref{def:FACS manifolds} can be extended to define the notion of \textit{a manifold with FACS whose boundary also has FACS} by allowing the conical neighborhoods to have base space \(B = \R^\hdim\) or \(B = \mathbb{H}^\hdim\).\footnote{\(\mathbb{H}^\hdim\) is the \(\hdim\)-dimensional half-space.}
    The same proof shows that Proposition \ref{prop:compactness theorem for facs manifolds} still holds in the context of FACS manifolds with FACS boundary.
    This observation will be useful when we prove the Rellich-Kondrachov compactness theorem for manifolds with iterated asymptotically conical singularities in Section \ref{subsec:functional analytic preliminaries - iacs manifolds}.
\end{remark}

To end this subsection, we sketch the proof of Fredholmness.
First, observe that if self-adjointness of \(D_B\) is proven, then since the domain of the operator compactly embeds into \(L^2\), a standard compact resolvent argument implies that the operator is Fredholm.

To prove self-adjointness, we use the observation that if self-adjointness can proven locally, then the operator must be self-adjoint globally.
To see this, take \(\sigma \in \operdom(D_B)^*\). \(\sigma\) can be written as \(\sigma=\sum_i\varphi_{U_i}\sigma\), where \(\{\varphi_{U_i}\}\) is a partition of unity subordinate to some finite open cover \(\{U_i\}\).
Each \(\varphi_{U_i} \sigma \in \operdom (D_B)^*\) since \(\sigma\) is.
If each \(\varphi_{U_i} \sigma\) is in \(\operdom (D_B)_{min}\), then \(\sigma\) must be as well.
Self-adjointness is standard for a neighborhood of a point in the smooth stratum, so it is only necessary to consider the case where \(U_i\) is a conical neighborhood.
Thus, after localizing the problem, we reduce to the case where the manifold is diffeomorphic to a direct product of a base space and a cone, and, geometrically, the manifold is an asymptotic family of cones over the base.
Using the Kato-Rellich perturbation theorem and Proposition \ref{prop:multiplication by 1/r is bounded}, it is possible to further reduce to the case where the metric is a direct product of a flat metric on the base with a cone metric.
Finally, self-adjointness in this direct product case is a consequence of self-adjointness for a cone, which we now prove.

\subsection{Self-adjointness for a cone}\label{sec:self-adjointness over point}
We consider the case of cone manifolds \((C\dom{L}^{\tdim}, dr^2 + r^2 \dom{g})\) and \((CL^{\tdim}, dr^2 + r^2 g)\) with links \((\dom{L}^{\tdim-1}, \dom{g})\) and \((L^{\tdim - 1}, g)\).
Here, \(\dom{L}\) and \(L\) are smooth compact manifolds without boundary.
The map \(F \colon C\dom{L} \to CL\) is assumed to be a smooth map of the form \(F(r, x) = (\varphi(r), f(x))\).

Recall Proposition \ref{prop:dirac operator under unitary transformation with grading}, which states that the Dirac operator on \(SV\) takes the form
\begin{equation}
    \Psi^{-1}D\Psi = \begin{pmatrix}
    0 & -1\\
    1 & 0
    \end{pmatrix}\partial_r + \frac{1}{r}\begin{pmatrix}
        0 & \mathcal{P}\\
        \mathcal{P} & 0\\
    \end{pmatrix}
\end{equation}
In general, if \(\mathcal{P}\) is diagonalizable with eigenvalues \(\{\lambda_i\}\), then
\begin{equation*}
    \begin{pmatrix}
        0 & -1\\
        1 & 0\\
    \end{pmatrix}\partial_r + \frac{1}{r}\begin{pmatrix}
        0 & \mathcal{P}\\
        \mathcal{P} & 0
    \end{pmatrix}
\end{equation*}
is self-adjoint if and only if \(|\lambda_i| \geq \frac{1}{2}\) for all \(i\), according to \cite{Bruning:1988}.
Using this, the next lemma allows us to conclude self-adjointness when the link has dimension at least two, that is, when \(n-1\) is at least 2.
\begin{lemma}\label{lem:link Dirac operator estimate}
    Suppose the curvature operator \(\mathcal{R}\) of \(CL^n\) is nonnegative.
    If \(\dom{\Sc} \geq \onorm{d F}^2 F^*\Sc\), then \(\abs{\mathcal{P}} \geq \frac{ \sqrt{ (n-1)(n-2) } }{2}\).
\end{lemma}
\begin{proof}
    The differential operator \(\mathcal{P}\) is given by
    \begin{equation*}
        \mathcal{P} = -\sum_i \cbar{\partial_r} \cbar{\dom{e}_i} \conn[\partial, 0]
    \end{equation*}
    where \(\{\dom{e}_i\}\) is a local orthonormal frame on \(\dom{L}\).
    Taking the square,
    \begin{equation*}
        \mathcal{P}^2 = (\conn[\partial, 0])^* \conn[\partial, 0] + \sum_{i < j} \cbar{\dom{e}_i} \cbar{\dom{e}_j} \curvtens[\partial, 0]_{\dom{e}_i, \dom{e}_j}
    \end{equation*}
    where \(\curvtens[\partial, 0]\) is the curvature tensor of \(SV|_{\dom{L}}\) viewed as a spinor bundle associated to \(T\dom{L} \oplus (F^*T\cone{L})|_{\dom{L}}\).
    By \cite[Theorem~4.15]{Lawson:1989},
    \begin{multline*}
        \curvtens[\partial, 0]_{\dom{e}_i, \dom{e}_j} \sigma = \frac{1}{2} \sum_{k < l} \inner{R_{\dom{e}_i, \dom{e}_j}^{T\dom{L}} \dom{e}_k, \dom{e}_l} \cbar{\dom{e}_k} \cbar{\dom{e}_l} \sigma \\
        + \frac{1}{2} \sum_{k < l} \binner{R_{f_*\dom{e}_i, f_*\dom{e}_j}^{TCL} \frac{e_k}{\varphi(1)}, \frac{e_l}{\varphi(1)}}_{\varphi(1)^2 g} c\Big(\frac{e_k}{\varphi(1)}\Big) c\Big(\frac{e_l}{\varphi(1)}\Big) \sigma
    \end{multline*}
    where \(\{e_k\}\) is a local orthonormal frame on \(TL\).
    The proof of Theorem \ref{thm:scalar curvature estimate} shows
    \begin{equation*}
        \sum_{i < j} \cbar{\dom{e}_i} \cbar{\dom{e}_j} \curvtens_{\dom{e}_i, \dom{e}_j}^{\partial, 0} \geq \frac{\Sc_{\dom{L}}}{4} - \varphi(1)^2 \onorm{\wedge^2 df} \frac{F^*\Sc_{CL}}{4} \notag
    \end{equation*}
    and thus
    \begin{align}
        \sum_{i < j} \cbar{\dom{e}_i} \cbar{\dom{e}_j} \curvtens_{\dom{e}_i, \dom{e}_j}^{\partial, 0} &\geq \frac{\Sc_{\dom{L}}}{4} - \varphi(1)^2 \onorm{df}^2 \frac{F^*\Sc_{CL}}{4} \notag \\
        \label{eqn:curvature estimate on cone} &\geq \frac{\Sc_{\dom{L}}}{4} - \frac{\onorm{df}^2}{4} \big( f^*\Sc_L - (n-1)(n-2) \big)
    \end{align}
    On the other hand, by hypothesis \(\Sc_{C\dom{L}} \geq \onorm{dF}^2 F^* \Sc_{CL}\), which implies
    \begin{align*}
        \Sc_{\dom{L}} - (n-1)(n-2) &\geq \max\{\varphi'(1), \varphi(1) \onorm{df}\}^2 \frac{f^*\Sc_L - (n-1)(n-2)}{\varphi(1)^2} \\
        &\geq \onorm{df}^2 \big(f^*\Sc_L - (n-1)(n-2)\big)
    \end{align*}
    Hence,
    \begin{equation*}
        \Sc_{\dom{L}} \geq \onorm{df}^2 \big(f^*\Sc_L - (n-1)(n-2)\big) + (n-1)(n-2).
    \end{equation*}
    Substituting the above into \eqref{eqn:curvature estimate on cone} gives
    \begin{equation*}
        \sum_{i < j} \cbar{\dom{e}_i} \cbar{\dom{e}_j} \curvtens_{\dom{e}_i, \dom{e}_j}^{\partial, 0} \geq \frac{(n-1)(n-2)}{4}.
    \end{equation*}

    To summarize,
    \begin{align*}
        \mathcal{P}^2 &= (\conn[\partial, 0])^* \conn[\partial, 0] + \sum_{i < j} \cbar{\dom{e}_i} \cbar{\dom{e}_j} \curvoper_{\dom{e}_i, \dom{e}_j}^{\partial, 0} \\
        &\geq (\conn[\partial, 0])^* \conn[\partial, 0] + \frac{(n-1)(n-2)}{4}.
    \end{align*}
    This implies \(\abs{\mathcal{P}} \geq \frac{\sqrt{(n-1)(n-2)}}{2}\).
\end{proof}
\begin{remark}\label{rem:enough to know link dirac operator is Fredholm}
    We used the standard elliptic operator theory of the smooth compact manifold \(\dom{L}\) to deduce that \(\mathcal{P}\) is diagonalizable.
    Recall that \(\mathcal{P}\) differs from the link Dirac operator \(D^{\partial, \partial}\) by a symmetric endomorphism (see Remark \ref{rem:link operator in terms of link dirac operator}).
    Thus, even if the link is not smooth compact without boundary, if it is known that \(D^{\partial, \partial}\) is self-adjoint, that the extra endomorphism term is bounded, and that a form of the Rellich-Kondrachov compactness theorem holds for the link, then \(\mathcal{P}\) must be diagonalizable.
    This observation will be important in Section \ref{sec:manifolds with iacs}, when we prove self-adjointness for manifolds with iterated asymptotically conical singularities by induction.
\end{remark}
\subsection{Self-adjointness in the general case}\label{subsec:self-adjointness in general case - facs manifolds}
Let \(\dom{b}_0\) and \(b_0\) be singular points of \(\dom{M}\) and \(M\) for which \(F(\dom{b}_0) = b_0\).
Take conical neighborhoods \(\dom{U} \cong \R^\hdim \times \cone[\dom{\varepsilon}]{\dom{L}}\) and \(U \cong \R^\hdim \times \cone[\varepsilon]{L}\) of the points \(\dom{b}_0\) and \(b_0\) such that \(F(\dom{U}) \subseteq U\) and
\begin{equation*}
    F(\dom{b},r,\dom{x}) = (\psi(\dom{b}), \varphi(\dom{b},r), f(\dom{b},\dom{x}))
\end{equation*}
in the trivialization.
The goal is to reduce to the case where the domain and target are just the direct product of a flat base with a cone.

In the trivialization, the metrics \(\dom{g}\) and \(g\) take the form
\begin{gather*}
    \dom{g}(\dom{b}) = \pi^*g_{\dom{B}} \oplus dr^2 \oplus r^2 g_{\dom{L}}(\dom{b}) + o_2(r^2)\\
    g(b) = \pi^*g_B \oplus dr^2 \oplus r^2 g_L(b) + o_2(r^2)
\end{gather*}
where \(\dom{B} = \R^\hdim = B\), with \(g_{\dom{B}}\) and \(g_B\) being some metrics on the base space, not necessarily the canonical one.
After rescaling by \(R^2\) and changing \(r\) to \(\gamma = Rr\), the metrics become
\begin{gather*}
    R^2 \dom{g}(\dom{b}) = R^2 \pi^*g_{\dom{B}} \oplus d\gamma^2 \oplus \gamma^2 g_{\dom{L}}(\dom{b}) + o_2(\gamma^2)\\
    R^2 g(b) = R^2 \pi^*g_B \oplus d\gamma^2 \oplus \gamma^2 g_L(b) + o_2(\gamma^2)
\end{gather*}
Intuitively, if the conical neighborhoods \(\dom{U}\) and \(U\) are made to be very small, then the metric on the base is nearly constant and the \(o_2(\gamma^2)\) error terms should be negligible.
Thus, it is natural to compare the original metrics with the following new metrics, defined for every \(R > 0\) by
\begin{gather*}
    \dom{g}_{\dom{b}_0}^R(\dom{b}) = R^2 \pi^*g_{\dom{B}}(\dom{b}_0) \oplus d\gamma^2 \oplus \gamma^2 g_{\dom{L}}(\dom{b}_0) \\
    g_{b_0}^R(b) = R^2 \pi^*g_B(b_0) \oplus d\gamma^2 \oplus \gamma^2 g_L(b_0)
\end{gather*}
It is also important to understand how the map \(F\) changes after transforming the radial coordinate.
If \(\gamma = Rr\), then
\begin{equation*}
    F(\dom{b},r,\dom{x}) = (\psi(\dom{b}), R\varphi(\dom{b},\gamma/R), f(\dom{b},\dom{x})).
\end{equation*}
Recall that, by definition, \(\varphi(\dom{b}, r) = a(\dom{b}) \cdot r + o_2(r)\).
In the limit as \(R \to \infty\), \(F\) will take the form
\begin{equation*}
    \widetilde{F}(\dom{b}, \gamma, \dom{x}) = (\psi(\dom{b}), a(\dom{b}) \cdot \gamma, f(\dom{b},\dom{x}))
\end{equation*}
Let \(\widetilde{F}_{b_0}\) denote the map
\begin{equation*}
    \widetilde{F}_{\dom{b}_0}(\dom{b}, \gamma, \dom{x}) = (\psi(\dom{b}), a(\dom{b}_0) \cdot \gamma, f_{\dom{b}_0}(\dom{x}))
\end{equation*}
where \(f_{\dom{b}_0}(\dom{x}) \coloneqq f(\dom{b}_0, \dom{x})\)

Define
\begin{align*}
    V^R &\coloneqq (T(\R^\hdim \times \cone[\dom{\varepsilon}]{\dom{L}}), R^2 \dom{g}) \oplus (T(\R^\hdim \times \cone[\varepsilon]{L}), R^2 g) \\
    \widetilde{V}^R &\coloneqq (T(\R^\hdim \times \cone[\dom{\varepsilon}]{\dom{L}}), \dom{g}_{\dom{b}_0}^R) \oplus (\widetilde{F}_{\dom{b}_0}^*T(\R^\hdim \times \cone[\varepsilon]{L}), \widetilde{F}_{\dom{b}_0}^*g_{b_0}^R).
\end{align*}
We start by checking that the Dirac operator on \(S\widetilde{V}^R\) is self-adjoint.
Since the new metrics are now just the direct sum of a flat base metric with a constant vertical cone metric and because the vertical map on cones is independent of the base, self-adjointness will be an immediate consequence of the results of Section \ref{sec:self-adjointness over point}, as long as it can be shown that the hypotheses of Lemma \ref{lem:link Dirac operator estimate} still hold.

The scalar curvatures \(\Sc_{\dom{g}_{\dom{b}_0}^R}\) and \(\Sc_{g_{b_0}^R}\) are given by Corollary \ref{cor:scalar curvature of cone}.
By assumption
\begin{equation*}
    \Sc_{\hat{g}}(\dom{b}_0) \geq \onorm{dF}^2(\dom{b}_0) F^*\Sc_g(b_0).
\end{equation*}
Hence,
\begin{equation*}
    \Sc_{ R^2 \hat{g} }(\dom{b}_0) \geq \onorm{d(F|_{\pi^{-1}\dom{b}_0})}^2 (F|_{\pi^{-1}\dom{b}_0})^*\Sc_{ R^2 g }(b_0).
\end{equation*}
In the limit as \(R \to \infty\), Proposition \ref{prop:asymptotic scalar curvature estimate for family of cones} and \cite[Proposition~2.17]{Streil:2016} imply
\begin{equation}\label{eqn:scalar curvature comparison - reduction to model case}
    \Sc_{\dom{g}_{\dom{b}_0}^R} \geq \onorm{d(\widetilde{F}_{\dom{b}_0}|_{\pi^{-1}\dom{b}_0})}^2 (\widetilde{F}_{\dom{b}_0}|_{\pi^{-1}\dom{b}_0})^*\Sc_{g_{b_0}^R}.
\end{equation}
The proof that the curvature operator of \(g_{b_0}^R\) is nonnegative is similar, except now one should apply Remark \ref{rmk:asymptotic sectional curvature for family of cones}, \cite[Proposition~2.8]{Streil:2016}, and \cite[Lemma~2.15]{Streil:2016}.

In conclusion, the Dirac operator on \(S\widetilde{V}^R\) is self-adjoint.

The last step to proving self-adjointness in the general case is to relate the twisted Dirac operator on \(SV\) to the one on \(S\widetilde{V}^R\) and then apply the Kato-Rellich perturbation theorem.
By \cite[pp.~7-8]{Bourguignon:1992}, for every \(R\) there is a unitary \( \Upsilon^R \colon S\widetilde{V}^R \to SV^R\).
Combining Equation \eqref{eqn:new connection matrix coefficient} from Proposition \ref{prop:asymptotic scalar curvature estimate for family of cones} with \cite[Theorem~3.1]{Bourguignon:1992} shows that all error terms of \((\Upsilon^R)^{-1} D_{V^R} \Upsilon^R - D_{\widetilde{V}^R}\) must either have asymptotic order \(o(1/\gamma)\) or leading asymptotic order of the form \(c(\dom{b})/\gamma\), where \(c(\dom{b})\) is a smooth function on the base converging to 0 when \(\dom{b}\) approaches \(\dom{b}_0\). By the proof of Proposition \ref{prop:multiplication by 1/r is bounded}, multiplication by \(1/\gamma\) is a bounded function from \(H^1\) to \(L^2\).
Taking \(R\) to be very big and shrinking \(\dom{U}\) and \(U\) to be very small, one can force
\begin{equation*}
    \norm{\left[ (\Upsilon^R)^{-1} D_{V^R} \Upsilon^R - D_{\widetilde{V}^R} \right] \sigma}_{L^2} \leq a \norm{ \sigma }_{L^2}^2 + b \norm{ D_{\widetilde{V}^R}\sigma }_{L^2}^2
\end{equation*}
with \(b < 1\).
By the Kato-Rellich perturbation theorem, the twisted Dirac operator with respect to some rescaling by \(R^2\) will be self-adjoint.
This implies that the original twisted Dirac operator coming from the unmodified metrics must be self-adjoint as well.
\begin{remark}\label{rem:self-adjointness can be proven with weaker lower bound on scalar curvature}
    To deduce self-adjointness, it is enough to assume that \(\Sc_{\dom{M}} \geq \Sc_M + C\), where \(C\) is some constant.
    In fact, \(C\) can be replaced by any function that behaves like \(o(1/r^2)\) near singular points.
    Indeed, when rescaling the metrics by \(R^2\) and doing the coordinate change \(\gamma = Rr\), after taking \(R \to \infty\), the scalar curvature comparison \eqref{eqn:scalar curvature comparison - reduction to model case} will still hold.
    This observation will be important when proving self-adjointness for manifolds with iterated asymptotically conical singularities.
\end{remark}
\section{Scalar-mean curvature comparison for manifolds with FACS}\label{sec:scalar-mean curvature comparison results}
Using the results of Sections \ref{sec:manifolds with facs} and \ref{sec:fredholmness}, we prove our main theorem in the case of manifolds with fiberwise asymptotically conical singularities.
\begin{theorem}
    Let \((\dom{M}^{\tdim},\bnd{\dom{M}}^{\tdim-1},\dom{g})\) and \((M^{\tdim},\bnd{M}^{\tdim-1},g)\) be compact spin manifolds with FACS and smooth boundaries \(\bnd{\dom{M}}\), \(\bnd{M}\).
    Assume that the links of \(\dom{M}\) and \(M\) have dimension at least two.
    Suppose that the curvature operator of \(M\) and the second fundamental form on the boundary of \(M\) are nonnegative.
    Let \(F \colon \dom{M} \to M\) be an asymptotically conical map, with \(F_\partial\) denoting its restriction to the boundary.
    Assume that one of the following conditions holds:
    \begin{enumerate}[label=\((\arabic*)\)]
        \item \label{itm:identical boundaries} The boundaries of \(\dom{M}\) and \(M\) coincide, namely, \(\bnd{\dom{M}} = \bnd{M}\) and \(g|_{\bnd{\dom{M}}}\) equals \(g|_{\bnd{M}}\).
        \item \label{itm:nonvanishing euler characteristic} The Euler characteristic of \(\bnd{M}\) is nonzero and \(F\) has nonzero degree.
    \end{enumerate}
    If the following comparison conditions hold:
    \begin{enumerate}[label=\textnormal{(\roman*)}]
        \item \(\Sc_{\dom{g}} \geq \onorm{dF}^2 F^*\Sc_g\)
        \item \(H_{\dom{g}} \geq \onorm{d(F_\partial)} (F_\partial)^*H_g\)
    \end{enumerate}
    then, we conclude
    \begin{enumerate}[label=\textup{(\Roman*)}]
        \item \(\Sc_{\dom{g}} = \onorm{dF}^2 F^* \Sc_g\)
        \item \(H_{\dom{g}} = \onorm{d(F_\partial)} (F_\partial)^*H_g\)
    \end{enumerate}
    Moreover,
    \begin{enumerate}[label=\((\mathrm{\Roman*}^\prime)\)]
        \item \label{itm:covering map} If \(\Ric_g > 0\), then \(\onorm{dF} \equiv a\) for some constant \(a > 0\) and \(F \colon (\dom{M}, a \cdot \dom{g}) \to (M,g)\) is a Riemannian covering map.
        \item \label{itm:flat implies ricci flat} If \((M, g)\) is flat, then \((\dom{M}, \dom{g})\) is Ricci flat.
    \end{enumerate}
\end{theorem}
\begin{proof}
    \textbf{Case \ref{itm:identical boundaries}.}

    Let \(D_B\) denote the twisted Dirac operator with absolute boundary condition \(B\).

    If \(D_B\) is not invertible, then from the usual Fredholm theory, there exists some nonzero \(\sigma \in \operdom(D_B)\) such that \(D_B \sigma = 0\).
    By Corollay \ref{cor:scalar-mean curvature estimate with boundary condition},
    \begin{multline*}
        0 = \int_{\dom{M}} \pnorm{D_B \sigma}^2 \geq \int_{\dom{M}} \pnorm{\nabla \sigma}^2 + \int_{\dom{M}} \frac{1}{4}(\Sc_{\dom{g}} - \onorm{dF}^2 F^*\Sc_g) \pnorm{\sigma}^2\\
        + \int_{ \bnd{\dom{M}} } \frac{1}{2}(H_{\dom{g}} - \onorm{d(F_\partial)} (F_\partial)^*H_g)\pnorm{\sigma}^2.
    \end{multline*}
    The above implies that \(\sigma\) is a parallel section, and hence must be nowhere vanishing and have pointwise constant norm.
    The only way in which the right-hand side of the above inequality could be 0 is if \(\Sc_{\dom{g}} = \onorm{dF}^2 \Sc_g\) and \(H_{\dom{g}} = \onorm{d(F_\partial)} (F_\partial)^*H_g\).

    Suppose now that \(D_B\) is invertible.
    Since the metrics coincide on the boundary and the map \(F\) is just the identity, the spinor bundle \(S(TM \oplus TM)|_{\bnd{M}}\) is canonically isomorphic to \(\Cl(TM)|_{\bnd{M}}\).
    If \(\nu \in \Cl(TM)|_{\bnd{M}}\) denotes the interior unit normal, then \(\nu\) satisfies the relative boundary condition and \(\conn[\partial, \partial] \nu = 0\) (see Remark \ref{rem:boundary conditions in case of clifford bundle}.)

    Extend \(\nu\) to a smooth section of \(\Cl(TM)\) in any way.
    Call this extension \(\widetilde{\psi}\).
    Since \(D_B\) is invertible, there exists \(\psi \in \operdom(D_B)\) such that \(D\psi = -D\widetilde{\psi}\).
    If we set \(\sigma = \psi + \widetilde{\psi}\), then \(D\sigma = 0\) and hence
    \begin{multline*}
        0 = \int_{\dom{M}} \pnorm{D\sigma}^2 \geq \int_{\dom{M}} \pnorm{\nabla \sigma}^2 + \int_{\dom{M}} \frac{1}{4}(\Sc_{\dom{g}} - \onorm{dF}^2 F^*\Sc_g) \pnorm{\sigma}^2 \\
        \int_{ \bnd{\dom{M}} } \left< D^{\partial, \partial} \sigma,\sigma \right> + \int_{ \bnd{\dom{M}} } \frac{1}{2}(H_{\dom{g}} - \onorm{d(F_\partial)} (F_\partial)^*H_g)\pnorm{\sigma}^2.
    \end{multline*}
    If the term with the boundary Dirac operator vanishes, then the conclusion will follow just as in the previous case.
    To this end, observe that
    \begin{align*}
        \int_{ \bnd{\dom{M}} } \left< D^{\partial,\partial} \sigma,\sigma \right> &= \int_{ \bnd{\dom{M}} } \left< D^{\partial,\partial} (\psi|_{\bnd{M}} + \nu), \psi|_{\bnd{M}} + \nu \right>\\
        &= \int_{ \bnd{\dom{M}} } \left< D^{\partial,\partial} \psi|_{\bnd{M}}, \nu \right>,\, \parbox[t]{4cm}{since \(D^{\partial,\partial}\) maps \(B\) to \(B^{\perp}\) and \(\conn[\partial, \partial] \nu\) = 0}\\
        &= \int_{ \bnd{\dom{M}} } \left< \psi|_{\bnd{M}}, D^{\partial,\partial} \nu \right> = 0.
    \end{align*}

    \textbf{Case \ref{itm:nonvanishing euler characteristic}.}

    The argument is similar to the previous case.
    Again, let \(D_B\) denote the twisted Dirac operator with absolute boundary condition \(B\).

    If \(D_B\) is not invertible, then the same proof as in Case \ref{itm:identical boundaries} establishes the scalar and mean curvature equalities.

    Suppose now that \(D_B\) is invertible.
    Since \(\chi(\bnd{M}) \ne 0\), the classical Atiyah-Singer index theorem applied to the boundary shows that there exists a section \(\widetilde{\Psi}_{\partial} \in S(T\dom{M} \oplus F^*TM)|_{\bnd{\dom{M}}}\) in the kernel of \(D^{\partial, \partial}\).
    If the projection of \(\widetilde{\Psi}_{\partial}\) onto \(B^{\perp}\) does not vanish, let \(\widetilde{\psi}_{\partial}\) be this projection and extend \(\widetilde{\psi}_{\partial}\) to a smooth section \(\widetilde{\psi}\) of the interior.
    Then, proceed as in Case \ref{itm:identical boundaries}.
    If the projection onto \(B^{\perp}\) vanishes, then \(\widetilde{\Psi}_{\partial}\) is absolute, which means that \(\widetilde{\psi}_{\partial} \coloneqq \cbar{\dom{\nu}} \widetilde{\Psi}_{\partial}\) is relative.
    Extend \(\widetilde{\psi}_\partial\) to a smooth section \(\widetilde{\psi}\) of the interior and proceed just as before.

    For statements \ref{itm:covering map} and \ref{itm:flat implies ricci flat}, the proofs are similar to that of \cite[p.~98]{Wang:2023}.
    We present them here for completeness.

    To prove \ref{itm:covering map}, we first prove that \(F\) is a homothety.
    Start by taking local orthonormal bases \(\{\dom{e}_i\}\) of \(\dom{M}\) and \(\{e_i\}\) of \(M\) such that \(F_* \dom{e}_i = \mu_i e_i\), \(\mu_i \geq 0\).
    The proof of scalar-mean curvature extremality combined with the proof of Theorem \ref{thm:scalar curvature estimate} shows not only that \(\dom{\Sc}_{\dom{g}} = \onorm{dF}^2\Sc_g\), but also
    \begin{equation*}
        \sum_{i,j} \Big(\onorm{dF}^2 - \frac{\mu_i^2\mu_j^2}{\onorm{\wedge^2 dF}}\Big)\curvtens[M](e_i,e_j,e_j,e_i) = 0.
    \end{equation*}
    On the other hand, \(\Ric_g > 0\) means that for every \(i\), there is some \(j\) for which \(\curvtens[M](e_i,e_j,e_j,e_i) > 0\).
    This implies that for every \(i\), there is some \(j\) such that
    \begin{equation*}
        \onorm{dF}^2 - \frac{\mu_i^2\mu_j^2}{\onorm{\wedge^2 dF}} = 0
    \end{equation*}
    and hence \(\onorm{dF}^2 = \onorm{\wedge^2 dF}\) and \(\mu_i = \onorm{dF}\) for every \(i\).
    Thus, \(F\) is a homothety.

    Define \(h \coloneqq \onorm{dF}\) and let \(U\) denote the open set \(\{h > 0\}\).
    Since \(F\) is a homothety, \(F^*g = h^2\dom{g}\) and hence on \(U\)
    \begin{equation*}
        F^*\Sc_g = \frac{\Sc_{\dom{g}}}{h^2} - \frac{2(n-1)}{h^3}\Delta h - \frac{(n-1)n}{h^4}\pnorm{dh}^2.
    \end{equation*}
    We have already shown \(\Sc_{\dom{g}} = h^2 F^*\Sc_g\), so
    \begin{equation*}
        \frac{2}{h^3}\Delta h = -\frac{n}{h^4}\pnorm{dh}^2
    \end{equation*}
    on \(U\) and hence
    \begin{equation*}
        2h^k\Delta h = -nh^{k-1}\pnorm{dh}^2
    \end{equation*}
    on \(M\), for all \(k \geq 2\).
    Additionally, on \(U \cap \bnd{\dom{M}}\),
    \begin{equation*}
        F^*H_g = \frac{H_{\dom{g}}}{h} - (n-1) \frac{dh(\dom{\nu})}{h^2}
    \end{equation*}
    where \(\dom{\nu}\) is the unit inner normal of \(\bnd{\dom{M}}\).
    It was shown that \(H_{\dom{g}} = \onorm{d(F_\partial)} (F_\partial)^*H_g\), but \(F\) is a homothety, so actually \(\dom{H}_{\dom{g}} = \onorm{dF} (F_\partial)^*H_g\).
    Since \(H_{\dom{g}} = hF^*H_g\), we must have \(dh(\dom{\nu}) = 0\) on \(U \cap \bnd{\dom{M}}\).
    Since \(h\) is \(0\) outside of \(U\), the equality \(dh(\dom{\nu}) = 0\) holds on \(\bnd{\dom{M}}\).

    By our assumption that \(F\) is asymptotically conical,
    \begin{equation*}
        F(\dom{b},r,\dom{x}) = (\psi(\dom{b}),\varphi(\dom{b},r),f(\dom{b},\dom{x}))
    \end{equation*}
    after trivializing in conical neighborhoods of the domain and target, where \(\varphi\) is of the form \(\varphi(\dom{b},r) = a(\dom{b}) \cdot r + o_2(r)\).
    This implies \(\varphi(\dom{b},r)'' = o(1/r)\) and \((\varphi(\dom{b},r) \onorm{df}/r)' = o(1/r)\).
    We conclude that \(h = O(1)\) and \(\frac{\partial h}{\partial r} = o(1/r)\) in a conical neighborhood.
    If now \(r(\dom{p}) \coloneqq \dist(\dom{p},\dom{\cat{C}})\), where \(\dom{\cat{C}}\) is the singular stratum of \(\dom{M}\), then
    \begin{equation*}\int_{r=\varepsilon} dh(\nabla r) h^k \dvol_{\dom{M}}=o(1/\varepsilon)\cdot \varepsilon^\ldim=o(\varepsilon^{\ldim-1})\to 0.\end{equation*}
    In particular, if Stokes's theorem is applied to \(\dive (h^k \nabla h)\), then the extra term coming from the singular stratum vanishes.
    Thus,
    \begin{align*}
        0 = \int_{\dom{M}} \dive (h^k \nabla h) &= \int_{\dom{M}} h^k \Delta h + \int_{\dom{M}} \inner{d(h^k), dh} \\
        &= -\frac{n}{2}\int_{\dom{M}} h^{k-1}\pnorm{dh}^2 + k\int_{\dom{M}} h^{k-1}\pnorm{dh}^2
    \end{align*}
    for all \(k \geq 2\).
    Taking \(k > \frac{n}{2}\) shows that \(\pnorm{dh} \equiv 0\) on \(\dom{M}\), so \(\norm{dF} = h\) is a constant.

    To prove \ref{itm:flat implies ricci flat}, recall that the proof of extremality showed that there exists a parallel section \(\sigma\).
    Furthermore, since \(M\) is flat, the curvature tensor of \(S(T\dom{M} \oplus F^*TM)\) depends only on the curvatures of \(\dom{M}\).
    Thus, if \(\{\dom{e}_i\}\) is an orthonormal basis of \(\dom{M}\), then
    \begin{align*}
        0 = \sum_{i} \dom{c}(\dom{e}_i) \curvtens[S(T\dom{M} \oplus F^*TM)]_{\dom{e}_i, \dom{e}_j} \sigma &= \frac{1}{4}\sum_{i,k,l} \curvtens[\dom{M}]_{ijkl} \dom{c}(\dom{e}_i)\dom{c}(\dom{e}_k)\dom{c}(\dom{e}_l) \sigma \\
        &= \frac{1}{4}\underbrace{\sum_{\substack{i,k,l \\ \text{distinct}}} \curvtens[\dom{M}]_{ijkl} \dom{c}(\dom{e}_i)\dom{c}(\dom{e}_k)\dom{c}(\dom{e}_l) \sigma}_{=0 \text{ by the Bianchi identity}} + \frac{1}{2} \sum_{i,l} \curvtens[\dom{M}]_{jiil} \dom{c}(\dom{e}_l) \sigma\\
        &= \frac{1}{2} c\big(\Ric_{\dom{g}}(\dom{e}_j)\big)\sigma
    \end{align*}
    for all \(j\).
    This implies \(\Ric_{\dom{g}} = 0\).
\end{proof}

\section{Manifolds with iterated asymptotically conical singularities}\label{sec:manifolds with iacs}
In this section, we generalize the notion of a manifold with fiberwise asymptotically conical singularities (FACS) to that of a manifold with iterated asymptotically conical singularities (IACS).
In particular, the notion of a \(\cat{C}_m\)-manifold will be defined.
In this new language, the manifolds with FACS, as defined in Definition \ref{def:FACS manifolds}, are the \(\cat{C}_1\)-manifolds.
After this, Fredholmness of the twisted Dirac operator in this more general setting will be established by generalizing the results of Section \ref{sec:fredholmness}.
Lastly, we will state Theorem \ref{thm:intro-scalar curvature comparison} precisely using our language of \(\cat{C}_m\)-manifolds.

\subsection{Geometric setup}
The \(\cat{C}_0\)-manifolds are defined to be the smooth manifolds, and a \(\cat{C}_0\)-map is just a smooth map between smooth manifolds.

The \(\cat{C}_m\)-manifolds are defined inductively.
\begin{definition}
    Let \(m \geq 1\) be an integer and \(M\) a topological space.
    Suppose that the \(\cat{C}_0\), ..., \(\cat{C}_{m-1}\)-manifolds and the notions of a \(\cat{C}_0\), ..., \(\cat{C}_{m-1}\)-map have been defined.

    We call a neighborhood \(U\) a \textit{\(\cat{C}_m\)-neighborhood} of \(p \in M\) if \(U\) is homeomorphic to a direct product of \(\R^\hdim\) with the cone \(\cone[\varepsilon]{L}\), where \(L\) is a \(\cat{C}_{m-1}\)-manifold without boundary, for which \(p\) is taken to a point of \(\R^\hdim \times \{0\}\) under the homeomorphism.
    The point \(p\) is a called a \textit{\(\cat{C}_m\)-singular point} of \(M\).

    A topological space \(M\) is said to be an \textit{\(\tdim\)-dimensional \(\cat{C}_m\)-manifold} if the following conditions hold:
    \begin{enumerate}[label=(\alph*)]
        \item \emph{Topology:} There is an open cover of \(M\) for which every open set in the cover is either homeomorphic to \(\R^\tdim\) or is a \(\cat{C}_k\)-neighborhood of some \(\cat{C}_k\)-singular point of \(M\), with \(k \leq m\).
        The transition functions between two nontrivially intersecting neighborhoods \(U\) and \(V\), both homeomorphic to \(\R^\tdim\), are required to be smooth.
        The transition functions between two nontrivially intersecting \(\cat{C}_k\)-neighborhoods \(U \cong \R^\hdim \times \cone[\varepsilon]{L}\) and \(V \cong \R^\hdim \times \cone[\varepsilon']{L'}\) are required to take the form:
        \begin{equation*}
            F(b,r,x) = (\psi(b), \varphi(b,r), f(b,x))
        \end{equation*}
        where \(b\) represents a point on the base space, \(r\) corresponds to the radial coordinate, and \(x\) is an element of the link.
        The functions \(\psi\) and \(\varphi\) are required to be smooth away from \(r=0\) and \(\varphi\) is assumed to take the form \(\varphi(b,r) = a(b) \cdot r + o_2(r)\) near \(r=0\), with \(a(b) > 0\) for all \(b\).
        The function \(f\) is required to be a \(\cat{C}_{m-1}\)-map.

        The \textit{\(\cat{C}_k\)-stratum} is the set of \(\cat{C}_k\)-singular points of \(M\).
        This set is denoted by \(\cat{C}_k\).
        The singular stratum is denoted by \(\cat{C}\) and is defined to be \(\cat{C} \coloneqq \cup_{k=1}^m \cat{C}_k\).
        The smooth stratum, or \(\cat{C}_0\)-stratum, is the complement \(M \setminus \cat{C}\) and is denoted by \(\cat{C}_0\).
        \item \emph{Geometry:} For each \(\cat{C}_k \neq \emptyset\) there is a metric \(g_{\cat{C}_k}\) on \(\cat{C}_k\).
        There is also a metric \(g\) defined on the smooth stratum of \(M\) such that in a \(\cat{C}_k\)-neighborhood \(U = \R^\hdim \times \cone[\varepsilon]{L}\) of a \(\cat{C}_k\)-singular point \(p\) the metric \(g\) takes form
        \begin{equation*}
            g|_U = \pi^*g_{\cat{C}_k}|_{\pi(U)} \oplus \beta'(b,r)^2 dr^2 \oplus \beta(b,r)^2 g_L(b) + o_2(r^2)
        \end{equation*}
        subject to the following criteria:
        \begin{enumerate}
            \item \(\pi^*g_{\cat{C}_k}|_{\pi(U)}\) is the pullback of the metric on the \(\cat{C}_k\)-stratum under the projection \(\pi \colon U \to \R^\hdim\).
            \item \(g_L(b)\) is a family of metrics on \(L\) such that the product \(\cat{C}_k|_{\pi(U)} \times L\) equipped with \(\pi^*g_{\cat{C}_k}|_{\pi(U)} \oplus g_L\) is a \(\cat{C}_{k-1}\)-manifold.
            \item \(\beta(b,r)\) is a real-valued function that is smooth away from \(r=0\) and is of the form \(\beta(b,r) = a(b) \cdot r + o_2(r)\) near \(r=0\), with \(a(b) > 0\) for all \(b\).
        \end{enumerate}
        To ease the notation, we will often write \(\pi^*g_B\) to denote the pullback of the base metric, instead of \(\pi^*g_{\cat{C}_k}|_{\pi(U)}\).
\end{enumerate}
\end{definition}
If \(M\) is a \(\cat{C}_m\)-manifold for some \(m\), then \(M\) is referred to as a \textit{manifold with iterated asymptotically conical singularities}, sometimes shortened to ``manifold with IACS" or just ``IACS manifold."

The same remarks following Definition \ref{def:FACS manifolds} apply.
We summarize them here.
In a conical neighborhood, after a reparametrization, the metric becomes a Riemannian submersion of a family of cones over the base, up to asymptotic error terms.
Two \(\cat{C}_k\)-singular points in the same connected component of \(\cat{C}_k\) must have the same link.
The notion of a manifold with IACS can naturally be extended to that of an \textit{IACS manifold with IACS boundary.}
The singular stratum \(\cat{C}\) may depend on the covering atlas.

We now define a natural class of maps between \(\cat{C}_m\)-manifolds.
Just as before, the definition will be inductive.
\begin{definition}\label{def:Cm map}
    Let \(m \geq 1\) be an integer and suppose the notions of a \(\cat{C}_0\), ..., \(\cat{C}_{m-1}\)-map have been defined.

    If \(\dom{M}\) and \(M\) are \(\cat{C}_m\)-manifolds, a map \(F\) is said to be a \(\cat{C}_m\)-map if
    \begin{enumerate}
        \item \(F\) maps \(\cat{C}_k\)-singular points to \(\cat{C}_k\)-singular points.
        \item \(F\) is smooth when restricted to the smooth stratum.
        \item If \(F\) maps a \(\cat{C}_k\)-neighborhood \(\dom{U} = \R^{\dom{\hdim}} \times \cone[\dom{\varepsilon}]{\dom{L}}\) into a \(\cat{C}_k\)-neighborhood \(U \cong \R^\hdim \times \cone[\varepsilon]{L}\), then \(\dom{\hdim} = h\), \(\dim \dom{L} = \dim L\), and in the trivialization the map \(F\) takes the form
            \begin{equation*}
                F(\dom{b},r,\dom{x}) = (\psi(\dom{b}), \varphi(\dom{b},r), f(\dom{b},\dom{x}))
            \end{equation*}
        where \(\dom{b}\) represents a point on the base space, \(r\) corresponds to the radial coordinate, and \(\dom{x}\) is an element of the link.
        The functions \(\psi\) and \(\varphi\) are required to be smooth away from \(r=0\) and \(\varphi\) is assumed to take the form \(\varphi(\dom{b},r) = a(\dom{b}) \cdot r + o_2(r)\) near \(r=0\), where \(a(\dom{b})>0\) for all \(\dom{b}\).
        The function \(f\) is required to be a \(\cat{C}_{k-1}\)-map.
    \end{enumerate}
\end{definition}

\subsection{Functional analytic preliminaries}\label{subsec:functional analytic preliminaries - iacs manifolds}
To prove the theorem, the Fredholmness of the twisted Dirac operator needs to be established first.
This will be deduced just as in Section \ref{sec:fredholmness}, by first establishing self-adjointness and then using the compactness of the inclusion of \(H^1\) into \(L^2\) to obtain a compact resolvent.
To this end, we generalize the results of Section \ref{subsec:functional analytic preliminaries}.

The notation used in this subsection is as follows:
\begin{definition} \mbox{}
    Let \(E\) denote a Hermitian vector bundle on an IACS manifold \(M\), possibly having a smooth boundary \(\bnd{M}\), with Hermitian connection.
    \begin{enumerate}
        \item \(C_c^\infty(M, E)\) is defined to be the set of smooth compactly supported sections of \(E\) vanishing near the singular stratum.
        \item \(H^1(M, E)\) is defined to be the completion of \(C_c^\infty(M, E)\) with respect to the norm
        \begin{equation*}
            \norm{\varphi}_{H^1}\coloneqq (\norm{\varphi}_{L^2}^2+\norm{\nabla\varphi}_{L^2}^2)^{1/2}
        \end{equation*}
    \end{enumerate}
\end{definition}

The first proposition of this subsection generalizes Proposition \ref{prop:multiplication by 1/r is bounded}.
\begin{proposition}\label{prop:multiplication by 1/r is bounded - iacs manifolds}
    Let \(E\) be a Hermitian vector bundle with Hermitian connection over a compact manifold \(M\) with IACS.
    Let \(d_\cat{C}(p) \coloneqq \dist(p, \cat{C})\) denote the distance from \(p\) to the closest singular point.
    If the links of \(M\) have dimension at least two, then multiplication by \(1/d_\cat{C}\) defines a bounded linear operator from \(H^1(M, E)\) to \(L^2(M, E) \).
    Moreover, the operator norm is bounded above by a constant that depends only on the geometry and topology of \(M\).
\end{proposition}
\begin{proof}
    We prove this by induction.
    This was proven in Proposition \ref{prop:multiplication by 1/r is bounded} for \(\cat{C}_1\)-manifolds.
    Suppose it holds for \(\cat{C}_1\)-manifolds, ..., \(\cat{C}_{m-1}\)-manifolds.
    To prove the statement for \(\cat{C}_m\)-manifolds it is sufficient to prove the statement in a \(\cat{C}_k\)-neighborhood of a \(\cat{C}_k\)-singular point for \(1 \leq k \leq m\).
    After localizing and accounting for error terms, the proof reduces to the case where \(M = \R^\hdim \times \cone{L}\) with model metric \(g = \pi^*g_B \oplus dr^2 \oplus r^2 g_L(b)\).
    Here, \(L\) is an \(\ldim\)-dimensional \(\cat{C}_k\)-manifold with \(k \leq m-1\).
    
    In the model case, \(d_\cat{C}(p) = \dist(p, \cat{C})\) is Lipschitz equivalent to \(\tilde{d}_\cat{C}(b,r,x) = r \dist(x, \cat{C}_{L(b)})\), where \(\dist(x, \cat{C}_{L(b)})\) denotes the closest distance from \(x\) to a singular point of the link \(L\) with respect to the metric \(g_L(b)\).
    Thus, it suffices to show that multiplication by \(1/\tilde{d}_\cat{C}\) is bounded.
    Let \(\sigma\) be a smooth compactly supported section vanishing near the singular stratum.
    Using the inductive hypothesis,
    \begin{align*}
        \norm{\frac{1}{\tilde{d}_\cat{C}} \sigma}_{L^2}^2 &= \int_B \int_0^\infty \int_L \frac{1}{r^2}\frac{1}{\dist(x, \cat{C}_{L(b)})^2} \pnorm{\sigma}^2 r^\ldim dvol_{L(b)} dr dvol_B \\
        & \leq \int_B \int_0^\infty \frac{C(b)}{r^2}\big(\norm{(\sigma(r)|_{L(b)})}_{L^2(L(b))}^2 + \norm{\nabla (\sigma(r)|_{L(b)})}_{L^2(L(b))}^2\big) r^\ldim dr dvol_B.
    \end{align*}
    We may assume that the \(C(b)\) are uniformly bounded, since otherwise we can just restrict to a smaller base.
    Next, observe that
    \begin{equation*}
        \int_B \int_0^\infty \frac{1}{r^2}\norm{(\sigma(r)|_{L(b)})}_{L^2(L(b))}^2 r^\ldim dr dvol_B \leq \frac{2}{l-1}\norm{\sigma}_{L^2(M)}^2
    \end{equation*}
    by essentially the same argument used in Proposition \ref{prop:multiplication by 1/r is bounded}.
    Finally,
    \begin{equation*}
        \int_B \int_0^\infty \frac{1}{r^2}\norm{\nabla (\sigma(r)|_{L(b)})}_{L^2(L(b))}^2 r^\ldim dr dvol_B \leq \norm{\nabla \sigma}_{L^2(M)}
    \end{equation*}
    by definition of \(\norm{\nabla \sigma}_{L^2(M)}\).

    Combining all of the above inequalities implies the conclusion.
\end{proof}

Next, we generalize Proposition \ref{prop:compactness theorem for facs manifolds}.
\begin{proposition}\label{prop:compactness theorem for iacs manifolds}
    Let \(E\) be a Hermitian vector bundle with Hermitian connection over a compact manifold \(M\) with IACS.
    If the links of \(M\) have dimension at least two, then the inclusion of \(H^1(M, E)\) into \(L^2(M, E)\) is compact.
\end{proposition}
\begin{proof}
    The proposition will be proved by induction.
    For this, it is useful to consider the more general statement where \(M\) is allowed to have a boundary with IACS.
    The case of \(\cat{C}_1\)-manifolds with \(\cat{C}_1\)-boundary was shown in Proposition \ref{prop:compactness theorem for facs manifolds} (see Remark \ref{rem:compactness theorem for C_1-manifolds with C_1-boundary}).
    Suppose the statement has been proven for compact \(\cat{C}_1\)-manifolds, ..., compact \(\cat{C}_{m-1}\)-manifolds, possibly with \(\cat{C}_1\)-boundary, ..., \(\cat{C}_{m-1}\)-boundary, respectively.

    Given a \(\cat{C}_m\)-manifold, take an exhaustion by compact \(\cat{C}_{m-1}\)-manifolds with \(\cat{C}_{m-1}\)-boundary.
    An identical argument to the one given in Proposition \ref{prop:compactness theorem for facs manifolds} shows that if a sequence \(\{\sigma_k\}\) is uniformly bounded in \(H^1\)-norm, then the \(L^2\)-norms of the restrictions of the \(\sigma_k\)'s to the complement of the exhaustion become uniformly small.
    By the inductive hypothesis, we may apply the compactness theorem to each IACS manifold in the exhaustion.
    The proof is then finished off using the same subsequence argument that occurs in Proposition \ref{prop:compactness theorem for facs manifolds}.
\end{proof}

\subsection{Fredholmness}\label{subsec:fredholmness for iacs manifolds}
Let \((\dom{M}^n, \bnd{\dom{M}}^{n-1}, \dom{g})\) and \((M^n, \bnd{M}^{n-1},g)\) be \(\cat{C}_m\)-manifolds with smooth boundaries \(\bnd{\dom{M}}\) and \(\bnd{M}\).
Let \(F \colon \dom{M} \to M\) be a \(\cat{C}_m\)-map.
Suppose \((M,g)\) has nonnegative curvature operator and \(\Sc_{\dom{g}} \geq \norm{dF}^2 \Sc_g + C\), with \(C\) some constant.
We write \(V\) to denote the direct sum \(T\dom{M} \oplus F^*TM\).

Let \(D_B\) denote the twisted Dirac operator on \(SV\) with domain \(C_c^\infty(M, SV; B)\).
Observe that \(\operdom(D_B)_{min} = H^1(M, SV; B)\).
That the minimal domain contains \(H^1(M, SV; B)\) is clear.
To see the inclusion \(\operdom(D_B)_{min} \subseteq H^1(M, SV; B)\), notice first that the same proofs of the results of Section \ref{subsec:geometry of family of asymptotic cones} show that the curvature operators behave like \(1/d_\cat{C}^2\) near the singular strati.
Combining this with Corollary \ref{cor:scalar-mean curvature estimate with boundary condition} and Proposition \ref{prop:multiplication by 1/r is bounded - iacs manifolds} shows that
\begin{equation*}
    \norm{D \sigma}^2_{L^2} \geq \norm{\conn[] \sigma}^2 - C' \norm{\sigma}_{L^2}^2
\end{equation*}
where \(C'\) is independent of \(\sigma\).
In particular,
\begin{equation*}
    \norm{\sigma}_{H^1} \leq C'(\norm{\sigma}_{L^2} + \norm{D \sigma}_{L^2})
\end{equation*}
for some \(C'\).

We will show that under the above hypotheses
\begin{equation*}
    D_B \colon H^1(\dom{M}, SV; B) \to L^2(\dom{M}, SV)
\end{equation*}
is Fredholm.

By Proposition \ref{prop:compactness theorem for iacs manifolds}, the domain of \(D_B\) is compactly embedded into \(L^2\), so to establish Fredholmness it is enough to prove self-adjointness.
This will be done by induction.

The case of \(\cat{C}_1\)-manifolds was proven in Section \ref{subsec:self-adjointness in general case - facs manifolds} (see Remark \ref{rem:self-adjointness can be proven with weaker lower bound on scalar curvature}).
Suppose the statement has been proven for compact \(\cat{C}_1\)-manifolds, ..., compact \(\cat{C}_{m-1}\)-manifolds.
To prove the statement for compact \(\cat{C}_m\)-manifolds, just as in the proof for the case of \(\cat{C}_1\)-manifolds, it is enough to show self-adjointness in a neighborhood of every \(\cat{C}_k\)-singular point.

Let \(\dom{b}_0\) and \(b_0\) be \(\cat{C}_k\)-singular points of \(\dom{M}\) and \(M\) for which \(F(\dom{b}_0) = b_0\).
Take \(\cat{C}_k\)-neighborhoods \(\dom{U} \cong \R^\hdim \times \cone[\dom{\varepsilon}]{\dom{L}}\) and \(U \cong \R^\hdim \times \cone[\varepsilon]{L}\) of the points \(\dom{b}_0\) and \(b_0\) such that \(F(\dom{U}) \subseteq U\) and
\begin{equation*}
    F(\dom{b},r,\dom{x}) = (\psi(\dom{b}), \varphi(\dom{b},r), f(\dom{b},\dom{x}))
\end{equation*}
in the trivialization.
In the trivialization, the metrics \(\dom{g}\) and \(g\) take the form
\begin{gather*}
    \dom{g}(\dom{b}) = \pi^*g_{\dom{B}} \oplus dr^2 \oplus r^2 g_{\dom{L}}(\dom{b}) + o_2(r^2)\\
    g(b) = \pi^*g_B \oplus dr^2 \oplus r^2 g_L(b) + o_2(r^2)
\end{gather*}
where \(\dom{B} = \R^\hdim = B\), with \(g_{\dom{B}}\) and \(g_B\) some metrics on the base space, not necessarily the canonical one.
After rescaling by \(R^2\) and changing \(r\) to \(\gamma = Rr\), the metrics become
\begin{gather*}
    R^2 \dom{g}(\dom{b}) = R^2 \pi^*g_{\dom{B}} \oplus d\gamma^2 \oplus \gamma^2 g_{\dom{L}}(\dom{b}) + o_2(\gamma^2)\\
    R^2 g(b) = R^2 \pi^*g_B \oplus d\gamma^2 \oplus \gamma^2 g_L(b) + o_2(\gamma^2)
\end{gather*}
Under this change, the map \(F\) becomes
\begin{equation*}
    F(\dom{b},r,\dom{x}) = (\psi(\dom{b}), R\varphi(\dom{b},\gamma/R), f(\dom{b},\dom{x})).
\end{equation*}
Recall that, by definition, \(\varphi(\dom{b}, r) = a(\dom{b}) \cdot r + o_2(r)\).
In the limit as \(R \to \infty\), \(F\) will take the form
\begin{equation*}
    \widetilde{F}(\dom{b}, \gamma, \dom{x}) = (\psi(\dom{b}), a(\dom{b}) \cdot \gamma, f(\dom{b},\dom{x}))
\end{equation*}
Let \(\widetilde{F}_{b_0}\) denote the map
\begin{equation*}
    \widetilde{F}_{\dom{b}_0}(\dom{b}, \gamma, \dom{x}) = (\psi(\dom{b}), a(\dom{b}_0) \cdot \gamma, f_{\dom{b}_0}(\dom{x}))
\end{equation*}
where \(f_{\dom{b}_0}(\dom{x}) \coloneqq f(\dom{b}_0, \dom{x})\)

Consider the new metrics, defined for every \(R > 0\) by
\begin{gather*}
    \dom{g}_{\dom{b}_0}^R(\dom{b}) = R^2 \pi^*g_{\dom{B}}(\dom{b}_0) \oplus d\gamma^2 \oplus \gamma^2 g_{\dom{L}}(\dom{b}) \\
    g_{b_0}^R(b) = R^2 \pi^*g_B(b_0) \oplus d\gamma^2 \oplus \gamma^2 g_L(b)
\end{gather*}
Define
\begin{align*}
    V^R &\coloneqq (T(\R^\hdim \times \cone[\dom{\varepsilon}]{\dom{L}}), R^2 \dom{g}) \oplus (T(\R^\hdim \times \cone[\varepsilon]{L}), R^2 g) \\
    \widetilde{V}^R &\coloneqq (T(\R^\hdim \times \cone[\dom{\varepsilon}]{\dom{L}}), \dom{g}_{\dom{b}_0}^R) \oplus (\widetilde{F}_{\dom{b}_0}^*T(\R^\hdim \times \cone[\varepsilon]{L}), \widetilde{F}_{\dom{b}_0}^*g_{b_0}^R).
\end{align*}
As in Section \ref{subsec:self-adjointness in general case - facs manifolds}, it is enough to prove that the twisted Dirac operator on \(S\widetilde{V}^R\) is self-adjoint.
Indeed, the twisted Dirac operator on \(SV^R\) will differ from the twisted Dirac operator on \(S\widetilde{V}^R\) by error terms that have leading asymptotic order \(o(1/d_\cat{C})\) or leading asymptotic order of the form \(c(\dom{b})/d_\cat{C}\), where \(d_\cat{C}(p) = \dist(p, \cat{C})\) is the distance to the singular stratum and \(c(\dom{b})\) approaches 0 as \(\dom{b}\) approaches \(\dom{b}_0\).
By the proof of Proposition \ref{prop:multiplication by 1/r is bounded - iacs manifolds}, multiplication by \(1/d_\cat{C}\) is bounded.
Thus, by taking \(R\) big enough and \(\dom{U}\) and \(U\) to be sufficiently small, we can get the twisted Dirac operator on \(SV^R\) as close as we like to the twisted Dirac operator on \(S\widetilde{V}^R\).
(For the precise details, look at Section \ref{subsec:self-adjointness in general case - facs manifolds}.)
By the Kato-Rellich perturbation theorem, the twisted Dirac operator on \(SV^R\) is self-adjoint, which implies that the original twisted Dirac operator we started with must be self-adjoint.

Let us show that the twisted Dirac operator \(D_{\widetilde{V}^R}\) on \(S\widetilde{V}^R\) is self-adjoint.
The proof of \eqref{eqn:scalar curvature comparison - reduction to model case} and the discussion that comes after shows that
\begin{equation}\label{eqn:scalar curvature comparison - reduction to model case - iacs manifolds}
    \Sc_{\dom{g}_{\dom{b}_0}^R} \geq \onorm{d(\widetilde{F}_{\dom{b}_0}|_{\pi^{-1}\dom{b}_0})}^2 (\widetilde{F}_{\dom{b}_0}|_{\pi^{-1}\dom{b}_0})^*\Sc_{g_{b_0}^R}
\end{equation}
and that the curvature operator of \(L\) is nonnegative.
In Lemma \ref{lem:link Dirac operator estimate}, it was shown that \eqref{eqn:scalar curvature comparison - reduction to model case - iacs manifolds} implies
\begin{equation*}
    \Sc_{\dom{L}} \geq \onorm{df_{\dom{b}_0}}^2 (f_{\dom{b}_0})^*\Sc_L - \onorm{df_{\dom{b}_0}}^2 (n-1)(n-2) + (n-1)(n-2).
\end{equation*}
It is not hard to see that \(\onorm{df_{\dom{b}_0}}\) is bounded above by some constant.
By the induction hypothesis, the link Dirac operator \(D^{\partial, \partial}\) is Fredholm.
Combining Propositions \ref{prop:multiplication by 1/r is bounded - iacs manifolds} and \ref{prop:compactness theorem for iacs manifolds} with Remark \ref{rem:enough to know link dirac operator is Fredholm}, the twisted Dirac operator on \(S\widetilde{V}^R\) must be self-adjoint.

\subsection{Scalar-mean curvature comparison for IACS manifolds}
We now give the precise statement of Theorem \ref{thm:intro-scalar curvature comparison} using the language of \(\cat{C}_m\)-manifolds.

\begin{theorem}\label{thm:scalar-mean curvature comparison - manifolds with iacs}
    Let \((\dom{M}^{\tdim},\bnd{\dom{M}}^{\tdim-1},\dom{g})\) and \((M^{\tdim},\bnd{M}^{\tdim-1},g)\) be compact spin \(\cat{C}_m\)-manifolds with smooth boundaries \(\bnd{\dom{M}}\), \(\bnd{M}\).
    Assume that the links of \(\dom{M}\) and \(M\) have dimension at least two.
    Suppose that the curvature operator of \(M\) and the second fundamental form on the boundary of \(M\) are nonnegative.
    Let \(F \colon \dom{M} \to M\) be a \(\cat{C}_m\)-map, with \(F_\partial\) denoting its restriction to the boundary.
    Assume that one of the following conditions holds:
    \begin{enumerate}[label=\((\arabic*)\)]
        \item The boundaries of \(\dom{M}\) and \(M\) coincide, namely, \(\bnd{\dom{M}} = \bnd{M}\) and \(g|_{\bnd{\dom{M}}}\) equals \(g|_{\bnd{M}}\).
        \item The Euler characteristic of \(\bnd{M}\) is nonzero and \(F\) has nonzero degree.
    \end{enumerate}
    If the following comparison conditions hold:
    \begin{enumerate}[label=\textnormal{(\roman*)}]
        \item \(\Sc_{\dom{g}} \geq \onorm{dF}^2 F^*\Sc_g\)
        \item \(H_{\dom{g}} \geq \onorm{d(F_\partial)} (F_\partial)^*H_g\)
    \end{enumerate}
    then, we conclude
    \begin{enumerate}[label=\textup{(\Roman*)}]
        \item \(\Sc_{\dom{g}} = \onorm{dF}^2 F^* \Sc_g\)
        \item \(H_{\dom{g}} = \onorm{d(F_\partial)} (F_\partial)^*H_g\)
    \end{enumerate}
    Moreover,
    \begin{enumerate}[label=\((\mathrm{\Roman*}^\prime)\)]
        \item If \(\Ric_g > 0\), then \(\onorm{dF} \equiv a\) for some constant \(a > 0\) and \(F \colon (\dom{M}, a \cdot \dom{g}) \to (M,g)\) is a Riemannian covering map.
        \item If \((M, g)\) is flat, then \((\dom{M}, \dom{g})\) is Ricci flat.
    \end{enumerate}
\end{theorem}
\begin{proof}
    Once Fredholmness has been established, the argument is exactly the same as the one presented in Section \ref{sec:scalar-mean curvature comparison results}.
\end{proof}

\section{Application of the main results}\label{sec:applications}
In this section, we apply the techniques established thus far to deduce some additional results regarding manifolds with iterated asymptotically conical singularities.

\subsection{Rigidity theorem for Euclidean domains}
We prove a rigidity theorem for Euclidean domains as a corollary of Theorem \ref{thm:scalar-mean curvature comparison - manifolds with iacs}, or rather its proof.
\begin{theorem}
    Let \((M^{\tdim},\bnd{M}^{\tdim-1},g)\) be a compact spin manifold with IACS.
    Suppose that the links of \(M\) have dimension at least two.
    Let \(F_\partial\colon\partial M\to\Sigma\) be a smooth map with nonzero degree, where \(\Sigma\) is a closed convex hypersurface in the Euclidean space $\R^n$.
    Assume that one of the following conditions holds:
    \begin{enumerate}[label=\((\roman*)\)]
        \item \(F_\partial\) is an isometry.
        \item \(n\) is odd.
    \end{enumerate}
    Let $H$ be the mean curvature of $\Sigma$. If \(\Sc_{g} \geq 0\) and \(H_{g} \geq \onorm{dF_\partial} H\), then \(g\) is flat.
\end{theorem}
\begin{proof}
    Let $\Omega$ be the Euclidean domain bounded by $\Sigma$.
    By definition, the singularities of \(M\) lie entirely within the interior.
    Therefore, the map \(F_\partial\) extends to a smooth map \(F\colon M\to\Omega\) with the same nonzero degree, mapping an open neighborhood of the singularities of \(M\) to some point in the interior of \(\Omega\).
    Since \(F\) is constant in a neighborhood of the singular stratum, the pullback \(F^*T\Omega\) must also be constant near the singular stratum.
    The proofs of Section \ref{sec:manifolds with iacs} go through to show that Fredholmness of the twisted Dirac operator and Theorem \ref{thm:scalar-mean curvature comparison - manifolds with iacs} still hold, even if the map \(F\) does not satisfy Definition \ref{def:Cm map}.
    Thus, $g$ is scalar-flat and \(H_g = \onorm{dF_\partial}H\)

    Recall that there are two cases in the proof of Theorem \ref{thm:scalar-mean curvature comparison - manifolds with iacs}. In Case (1), there exists a nonzero section \(\sigma\) in \(H^1(M,S(TM\oplus F^*T\Omega))\) such that \(\sigma\) is parallel and satisfies the absolute boundary condition \(B\). 
    In Case (2), we also obtain a nonzero parallel section $\sigma$ in \(H^1(M,S(TM\oplus F^*T\Omega))\), though $\sigma$ may not satisfy either absolute or relative boundary condition.
    Nevertheless, in either case, we obtain a parallel section \(\sigma\) which not only satisfies Inequality \eqref{eqn:scalar-mean curvature estimate with boundary condition}, but also turns it into an equality.
    
    For local orthonormal bases $\{\hat e_1,\ldots,\hat e_n\}$ of $TM$ and $\{e_1,\ldots,e_n\}$ of $T\Omega=\R^n$, let us denote
    $$\gamma_1=(\sqrt{-1})^{n(n+1)/2} \hat c(\hat e_1)\hat c(\hat e_2)\cdots\hat c(\hat e_n)\textup{ and }\gamma_2=(\sqrt{-1})^{n(n+1)/2} c(e_1)c(e_2)\cdots c(e_n).$$
    Then $\gamma_1$ and $\gamma_2$ are both self-adjoint involutions. 
    Let us also fix a point $x\in \partial M$ such that $\Sigma$ is strictly convex at $F_\partial(x)$. Suppose that $\{e_1,e_2,\ldots,e_{n-1}\}$ is an orthonormal basis of $T_{F_\partial(x)}\Sigma$. The proof of Proposition \ref{prop:scalar-mean curvature estimate} combined with the positivity of the second fundamental form of $\Sigma$ at $F_\partial(x)$ shows that $F_\partial$ is a homothety at \(x\) and that the section $\sigma$ satisfies
    $$\hat c_\partial(w) c_\partial(dF_\partial(w))\sigma(x) =C\sigma(x)$$
    for any vector $w$ that is tangent to $T_x\partial M$. 
    Here, \(C > 0\) is some constant and
    \begin{equation*}
        \dom{c}_\partial(w) \coloneqq \dom{c}(\dom{\nu}) \dom{c}(w) \text{ and } c_\partial(dF_\partial(w)) \coloneqq c(\nu) c(dF_\partial(w))
    \end{equation*}
    where \(\dom{\nu}\) and \(\nu\) are the unit inner normal vectors of $\partial M$ at $x$ and $\Sigma$ at $F_\partial(x)$, respectively.
    Indeed, using \(H_g = \onorm{dF_\partial} H\) and the inequalities established in Proposition \ref{prop:scalar-mean curvature estimate}, we deduce
    \begin{equation*}
        \sum_i \Big(\onorm{dF_\partial} - \frac{\mu_i^2}{\onorm{dF_\partial}}\Big) \inner{\mathcal{A}e_i, e_i} = 0.
    \end{equation*}
    Strict convexity at \(x\) means \(\inner{\mathcal{A}e_i, e_i} > 0\) for each \(i\), which implies \(\mu_i = \onorm{dF_\partial}\).
    Since \(F\) is a homothety, we may perform a change of bases to get \(\{\dom{e}_i\}\) and \(\{e_i\}\) local orthonormal bases of \(\bnd{M}\) and \(\Sigma\) such that \(dF_\partial (\dom{e}_i) = \onorm{dF_\partial} e_i\) and \(\{e_i\}\) diagonalizes \(\mathcal{A}\).
    Denote \(\beta_i = \inner{\mathcal{A} e_i, e_i}\).
    The proof of Proposition \ref{prop:scalar-mean curvature estimate} combined with our geometric assumptions shows
    \begin{equation*}
        \int_{\bnd{M}} \inner{\big(\frac{1}{2} H_g + \frac{1}{2} \sum_{i,j} \inner{\conn[]_{F_*\dom{e}_i} \nu, e_j} \dom{c}_\partial(\dom{e}_i) c_\partial(e_j)\big) \sigma, \sigma} = 0
    \end{equation*}
    which translates to
    \begin{equation*}
        \int_{\bnd{M}} \sum_i \frac{1}{2} \inner{\big(\onorm{dF_\partial} \beta_i - \onorm{dF_\partial} \beta_i \dom{c}_\partial(\dom{e}_i) c_\partial(e_i)\big) \sigma, \sigma} = 0.
    \end{equation*}
    Clearly,
    \begin{equation*}
        \onorm{dF_\partial} \beta_i - \onorm{dF_\partial} \beta_i \dom{c}_\partial(\dom{e}_i) c_\partial(e_i) \geq 0
    \end{equation*}
    pointwise for each \(i\).
    Hence,
    \begin{equation*}
        \dom{c}_\partial(\dom{e}_i) c_\partial(dF_\partial(e_i)) \sigma(x) = \onorm{dF_\partial}\sigma(x).
    \end{equation*}
     
    We shall prove that $g$ is flat by constructing sufficiently many parallel sections of $S(TM\oplus F^*T\Omega)$, following the same idea as the proof in \cite{WangXieflat}.
    Note first that the self-adjoint operator $\hat c_\partial(w) c_\partial(dF_\partial(w))$ commutes with both $\gamma_1$ and $\gamma_2$.

    \textbf{Case (I).} $n$ is even.
    
    In this case, by projecting \(\sigma\) onto the \(\pm 1\)-eigenspaces of \(\gamma_1\), we may assume, without loss of generality, that the nonzero parallel section $\sigma$ lies in $S(TM\oplus F^*T\Omega)^{+\bullet}$, where $S(TM\oplus F^*T\Omega)^{+\bullet}$ is the subbundle of $S(TM\oplus F^*T\Omega)$ given by the $+1$-eigenspace of $\gamma_1$. 
    
    Let us define
   \begin{equation*}
        \mathscr V=\{\lambda=e_{i_1}e_{i_2}\cdots e_{i_k}:~0\leq k\leq n-1,~i_1<i_2<\cdots<i_k\},
    \end{equation*}
    considered as an orthonormal subset of $\Cl(\R^n)$, and denote
    $$c(\lambda)\coloneqq c(e_{i_1})c(e_{i_2})\cdots c(e_{i_k}).$$
    Note that for any $\lambda\in\mathscr V$, $c(\lambda)$ commutes with $\gamma_1$. Thus, $c(\lambda)\sigma$ is still a parallel section in $S(TM\oplus F^*T\Omega)^{+\bullet}$.
    
    We claim that $\{c(\lambda)\sigma\}_{\lambda\in\mathscr V}$ is an orthonormal parallel basis of $S(TM\oplus F^*T\Omega)^{+\bullet}$.
    Indeed, for any nontrivial $\lambda=e_{i_1}e_{i_2}\cdots e_{i_k}\in\mathscr V$, let us pick $w\in T_x\partial M$ so that $dF_\partial(w)=e_{i_1}$.
    Direct computation shows that
    $$\hat c_\partial(w) c_\partial(dF_\partial (w))\textup{ anti-commutes with } c(\lambda).$$
    This implies that \(c(\lambda) \sigma\) and \(\sigma\) are orthogonal at \(x\). As both \(c(\lambda)\sigma\) and \(\sigma\) are parallel, they must actually be orthogonal on all of \(M\). 
    Using this, we conclude that $\langle c(\lambda_1)\sigma,c(\lambda_2)\sigma\rangle=0$ for any $\lambda_1\ne\lambda_2\in\mathscr V$.
    Therefore, the bundle \(S(TM\oplus F^*T\Omega)^{+\bullet}\) over \(M\) admits \(2^{n-1}\) nonzero parallel sections, which is equal to its rank.
    This means that the curvature form of this bundle must vanish, implying that \(g\) is flat. 

    \textbf{Case (II).} $n$ is odd.

    In this case, let us consider the subbundle $S(TM\oplus F^*T\Omega)^{\bullet+}$ given by the $+1$-eigenspace of $\gamma_2$. Once again, we may assume, without loss of generality, that $\sigma$ is a section of $S(TM\oplus F^*T\Omega)^{\bullet+}$. Let $\mathscr V$ be as above.
    Note that for any $\lambda\in\mathscr V$, $c(\lambda)$ commutes with $\gamma_2$. Thus, $c(\lambda)\sigma$ is still a parallel section in $S(TM\oplus F^*T\Omega)^{\bullet+}$.
    
    We claim that $\{c(\lambda)\sigma\}_{\lambda\in\mathscr V}$ forms an orthonormal parallel basis of $S(TM\oplus F^*T\Omega)^{\bullet+}$.
    Indeed, for any nontrivial $\lambda=e_{i_1}e_{i_2}\cdots e_{i_k}\in\mathscr V$, let us pick $w\in T_x\partial M$ so that $dF_\partial(w)=e_{i_1}$. Just as in the previous case, a direct computation shows
    $$\hat c_\partial(w) c_\partial(dF_\partial (w))\textup{ anti-commutes with } c(\lambda).$$
    This implies that \(c(\lambda) \sigma\) and \(\sigma\) are orthogonal at \(x\), and hence are orthogonal on all of \(M\).
    Using this, we conclude that $\langle c(\lambda_1)\sigma,c(\lambda_2)\sigma\rangle=0$ for any $\lambda_1\ne\lambda_2\in\mathscr V$.    
    Therefore, the bundle \(S(TM\oplus F^*T\Omega)^{\bullet+}\) over \(M\) admits \(2^{n-1}\) nonzero parallel sections, which is equal to its rank. Hence, \(g\) is flat. 
\end{proof}

\subsection{Positive mass theorem}
We start by establishing some definitions and notation.

\begin{definition}
    Consider the Euclidean flat space \(\R^n\) with a smooth positive function \(\rho\) that is equal to the distance function from the origin outside a compact set.
    For any \(\delta > 0\) and any measurable function \(f\), define
    \begin{equation*}
        \norm{f}_{L_{\delta}^2}^2 \coloneqq \int_{\R^n} \abs{f}^2 \rho^{-2 \delta - n}.
    \end{equation*}
    The \(\delta\)-weighted \(L^2\)-space is defined to be
    \begin{equation*}
        L_{\delta}^2(\R^n) = \{f \colon \R^n \to \C : \norm{f}_{L_{\delta}^2} < \infty\}.
    \end{equation*}
    endowed with the \(\norm{\cdot}_{L_\delta^2}\) norm.

    The \(\delta\)-weighted \(H^k\)-Sobolev space is defined to be
    \begin{equation*}
        H_{\delta}^k(\R^\tdim) = \{f \colon \R^n \to \C : \partial^{\alpha} f \in L_{\delta - \abs{\alpha}}^2(\R^n) \;\forall \text{ multi-indices \(\alpha\) with \(\abs{\alpha} \leq k\)}\}
    \end{equation*}
    where differentiation is in the weak sense.
    This space is endowed with the norm
    \begin{equation*}
        \norm{f}_{H^k_\delta}^2 = \sum_{\abs{\alpha} \leq k} \norm{\partial^\alpha f}_{L_{\delta - \abs{\alpha}}^2}^2.
    \end{equation*}
\end{definition}
It is not hard to see that the vector spaces \(L_{\delta}^2(\R^\tdim)\) and \(H_\delta^k(\R^\tdim)\), as well as the topologies induced by the norms, do not depend on the choice of \(\rho\).

\begin{definition}
    A \(\cat{C}_k\)-manifold \((M,g)\) is said to be asymptotically flat if the following conditions hold:
    \begin{enumerate}[label=(\arabic*)]
        \item There is a compact set \(K\) such that \((K,g)\) is a \(\cat{C}_k\)-manifold with smooth boundary.
        \item There exists a diffeomorphism \(\Phi \colon M \setminus K \to \R^n \setminus K'\), where \(K'\) is a compact subset of \(\R^n\), satisfying the following properties:
        \begin{enumerate}
            \item \(\Phi_* g - g_{\R^n} \in H_\tau^2(\R^n \setminus K')\) where \(\tau = \frac{n-2}{2}\)
            \item \(\lambda^{-1} g_{\R^n} \leq \Phi_*g \leq \lambda g_{\R^n}\) for some \(\lambda > 0\)
        \end{enumerate}

        \item \(\Sc_g \in L^1(M \setminus K)\)
    \end{enumerate}
\end{definition}

For an asymptotically flat \(\cat{C}_k\)-manifold \((M,g)\) the ADM mass is defined in the usual manner, namely,
\begin{equation*}
    m(g) = \frac{1}{16\pi}\lim_{R \to \infty} \int_{S(R)} \big( (\Phi_*g)_{ij,j} - (\Phi_*g)_{jj,i} \big) dS^i
\end{equation*}
where \((\Phi_*g)_{ij,k}\) denotes the \(k\)-th partial derivative, \(S(R)\) is the standard round sphere of radius \(R\), and \(dS^i\) is the normal surface element in the direction of \(i\).

We will prove the positive mass theorem for IACS manifolds:
\begin{theorem}\label{thm:positive mass theorem for iterated asymptotically conical spaces}
    Let \((M^\tdim,g)\) be an \(\tdim\)-dimensional asymptotically flat spin \(\cat{C}_k\)-manifold.
    Suppose that the links of \(M\) have dimension at least two.
    If \(\Sc_g \geq 0\), then \(m(g) \geq 0\).
    Moreover, if \(m(g) = 0\), then the metric \(g\) is flat.
\end{theorem}

Before getting into the proof, we first extend the notion of a \(\delta\)-weighted space to include Hermitian bundles with Hermitian connection on asymptotically flat IACS manifolds.
\begin{definition}
    Let \(E\) be a Hermitian bundle with Hermitian connection on an asymptotically flat IACS manifold \(M\).

    The \(\delta\)-weighted \(H^k\)-Sobolev space \(H_\delta^k(M,E)\) is defined to be the completion of the smooth compactly supported sections of \(E\) with respect to the norm
    \begin{equation*}
        \norm{\sigma}_{H_\delta^k}^2 \coloneqq \sum_{\alpha \leq k} \norm{\conn[\alpha] \sigma}_{L_{\delta-\alpha}^2}^2
    \end{equation*}
    where
    \begin{equation*}
        \norm{\conn[\alpha] \sigma}^2 \coloneqq \int_M \pnorm{\conn[\alpha] \sigma}^2 \rho^{-2\delta-n}.
    \end{equation*}
    Here, \(\rho\) is any positive function on \(M\) that is equal to the Euclidean distance from the origin outside some compact set \(K\) for which \(M \setminus K \cong \R^\tdim \setminus K'\).
\end{definition}

The proof of Theorem \ref{thm:positive mass theorem for iterated asymptotically conical spaces} utilizes the following crucial lemma.
\begin{lemma}\label{lem:invertibility of dirac operator}
    If \(\Sc_g \geq 0\), then the Dirac operator
    \begin{equation*}
        D \colon H_{-\tau}^1(M, SM) \to L_{-\tau-1}^2(M, SM)
    \end{equation*}
    is invertible when \(\tau = \frac{n-2}{2}\).
\end{lemma}
\begin{proof}
    The proofs of Section \ref{subsec:fredholmness for iacs manifolds} carry over to show that self-adjointness holds near the singular stratum.
    Combining this with results of \cite{Bartnik:1986} we get \(D^* = D\), where on the left-hand side \(D^*\) is interpreted as the functional-analytic adjoint of
    \begin{equation*}
        D \colon H_{-\tau}^1(M, SM) \to L_{-\tau-1}^2(M, SM)
    \end{equation*}
    and on the right-hand side \(D\) is interpreted as the map
        \begin{equation*}
        D \colon H_{\tau-n+1}^1(M, SM) \to L_{\tau-n}^2(M, SM).
    \end{equation*}

    The proof of Fredholmness also carries over from Section \ref{subsec:fredholmness for iacs manifolds}, allowing us to construct a parametrix for \(D\) on some compact set \(K\), with smooth boundary, containing the singular stratum.
    The results of \cite{Bartnik:1986} allow us to construct a parametrix for \(D\) on the infinite end \(M \setminus K\).
    By gluing these parametrices together with a partition of unity, a parametrix for \(D\) on the entire \(\cat{C}_k\)-manifold \(M\) can be constructed.
    Hence, \(D\) is Fredholm, and in particular has closed image.

    To show that the operator is invertible, we start by showing that the kernel is trivial.
    By definition, if \(\sigma \in H_{-\tau}^1\), then
    \begin{equation*}
        \int_M \pnorm{\conn[] \sigma}^2 < \infty \text{ and } \int_M \pnorm{\sigma}^2 \rho^{-2} < \infty.
    \end{equation*}
    If \(\sigma\) were smooth and compactly supported, then, by the Lichnerowicz formula,
    \begin{equation*}
        \int_M \frac{\Sc_g}{4} \pnorm{\sigma}^2 = \int_M \pnorm{D\sigma}^2 - \int_M \pnorm{\conn[] \sigma}^2
    \end{equation*}
    where the right-hand side is bounded above by a constant multiple of \(\norm{\conn[] \sigma}_{L^2}^2\)
    Since the smooth compactly supported functions are dense in \(H_{-\tau}^1(M, SM)\), multiplication by \(\frac{\sqrt{\Sc_g}}{2}\) must a bounded linear map from \(H_{-\tau}^1(M, SM)\) to \(L^2(M, SM)\).
    Thus, the Lichnerowicz formula holds for general \(\sigma \in H_{-\tau}^1(M, SM)\).

    Suppose \(\sigma \in H_{-\tau}^1(M, SM)\) were a nonzero element in the kernel of \(D\).
    By ellipticity, \(\sigma\) must be a smooth section of \(SM\) away from the singular stratum.
    Combining the nonnegativity of the scalar curvature with the Lichnerowicz formula, we see that \(\sigma\) is parallel.
    In particular, \(\pnorm{\sigma}\) must be a nonzero constant.
    This implies that for some large \(R > 0\),
    \begin{equation*}
        \int_M \pnorm{\sigma}^2 \rho^{-2} \geq \int_{R}^{\infty} \pnorm{\sigma}^2 \rho^{n-3} d\rho = \infty
    \end{equation*}
    a contradiction.
    Hence, the kernel of \(D\) must be trivial.

    To establish surjectivity, we investigate the kernel of
    \begin{equation*}
        D \colon H_{\tau - n + 1}^1(M, SM) \to L_{\tau - n}^2(M, SM).
    \end{equation*}
    By definition, an element \(\sigma^* \in H_{\tau + 1 - n}^1(M, SM) = H_{-n/2}^1(M, SM)\) satisfies
    \begin{equation*}
        \int_M \pnorm{\conn[] \sigma^*}^2 \rho^2 < \infty \text{ and } \int_M \pnorm{\sigma^*}^2 < \infty.
    \end{equation*}
    In particular, \(\conn[] \sigma^*\) is in \(L^2(M, SM)\).
    The same argument as before implies that multiplication by \(\frac{\sqrt{\Sc_g}}{2}\) is a bounded linear map from \(H_{\tau + 1 - n}^1(M, SM)\) to \(L^2(M, SM)\).
    Thus, the Lichnerowicz formula holds for \(\sigma \in H_{\tau + 1 - n}^1(M, SM)\).

    Now, suppose \(\sigma^* \in H_{\tau+1-n}^1(M, SM)\) were a nonzero element in the kernel of \(D^*\).
    By ellipticity and the Lichnerowicz formula, \(\sigma^*\) is smooth away from the singular stratum and also parallel.
    Hence, \(\pnorm{\sigma^*}\) must be a nonzero constant.
    This implies that for sufficiently large \(R > 0\),
    \begin{equation*}
        \int_M \pnorm{\sigma^*}^2 \geq \int_R^\infty \pnorm{\sigma^*}^2 \rho^{n-1} d\rho = \infty
    \end{equation*}
    a contradiction.
\end{proof}
\begin{proof}[Proof of Theorem \ref{thm:positive mass theorem for iterated asymptotically conical spaces}]
    Let \(E = M \setminus K\).
    Let \(\sigma_0\) denote the pullback of a constant spinor on \(E\).
    Smoothly extend \(\sigma_0\) to a section of \(M\) vanishing near the singular stratum.
    According to \cite{Bartnik:1986}, \(\conn[] \sigma_0 \in L_{-\tau-1}^2(M, SM)\).
    By Lemma \ref{lem:invertibility of dirac operator}, there exists a unique \(\sigma_1 \in H_{-\tau}^1(M, SM)\) for which \(D\sigma_1 = -D\sigma_0\).
    Let \(\sigma = \sigma_1 + \sigma_0\).
    Then, \(D\sigma = 0\) and \(\sigma = \sigma_0\) modulo \(H_{-\tau}^1(M, SM)\).
    In particular, \(\conn[] \sigma \in L_{-\tau-1}^2(M, SM)\).

    By the ellipticity of \(D\), \(\sigma_1\) is smooth except at the singular stratum.
    The proof of Lemma \ref{lem:invertibility of dirac operator} showed that the Lichnerowicz formula holds for \(H_{-\tau}^1(M, SM)\), namely,
    \begin{equation}\label{eqn:lichnerowicz formula for H1_tau section}
        \int_M \pnorm{D\sigma_1}^2 = \int_M \pnorm{\nabla \sigma_1}^2 + \int_M \frac{\Sc_g}{4}\pnorm{\sigma_1}^2.
    \end{equation}
    Take a compactly supported cutoff function \(\psi\) which is equal to one near the singular stratum.
    Since \(\psi \sigma_1 \in H_{-\tau}^1(M, SM)\), again we must have
    \begin{equation*}
        \int_M \pnorm{D (\psi \sigma_1)}^2 = \int_M \pnorm{\nabla (\psi \sigma_1)}^2 + \int_M \frac{\Sc_g}{4}\pnorm{\psi \sigma_1}^2.
    \end{equation*}
    On the other hand, a direct application of the Lichnerowicz formula and Stokes's theorem gives
    \begin{equation*}
        \int_M \pnorm{D (\psi \sigma_1)}^2 = \int_M \pnorm{\nabla (\psi \sigma_1)}^2 + \int_M \frac{\Sc_g}{4}\pnorm{\psi \sigma_1}^2 + \lim_{R \to 0} \int_{d_\cat{C}=R} \big< L^\partial \sigma_1, \sigma_1 \big>.
    \end{equation*}
    Here, \(d_\cat{C}(p) \coloneqq \dist(p, \cat{C})\) is the distance to the singular stratum and \(L^\partial\) is the differential operator on the boundary defined by
    \begin{equation*}
        L^\partial \sigma = \sum_{i=1}^{n-1} c(\nu)c(e_i)\conn[]_{e_i} \sigma
    \end{equation*}
    where \(\nu\) is the inner unit normal vector and \(\{\nu, \{e_i\}\}\) is an orthonormal frame.
    Thus,
    \begin{equation}\label{eqn:contribution from singularity vanishes}
        \lim_{R \to 0} \int_{d_\cat{C}=R} \big< L^\partial \sigma_1, \sigma_1 \big> = 0.
    \end{equation}
    A direct application of the Lichnerowicz formula, Stokes's theorem, and \eqref{eqn:contribution from singularity vanishes} gives
    \begin{equation}
        \label{eqn:direct lichnerowicz formula}\int_{D(R)} \pnorm{D\sigma_1}^2 = \int_{D(R)} \pnorm{\conn[]\sigma_1}^2 + \int_{D(R)} \frac{\Sc_g}{4} \pnorm{\sigma_1}^2 + \int_{S(R)} \big< L^\partial \sigma_1, \sigma_1 \big>
    \end{equation}
    for large balls \(D(R)\).
    Taking the limit as \(R \to \infty\), since \eqref{eqn:lichnerowicz formula for H1_tau section} must hold, we conclude
    \begin{equation*}
        \lim_{R \to \infty} \int_{S(R)} \big< L^\partial \sigma_1, \sigma_1 \big> = 0.
    \end{equation*}

    Observe that if \(\psi\) is a smooth compactly supported section, then
    \begin{align*}
        \abs{\int_M \frac{\Sc_g}{4} \inner{\psi, \sigma_0}} \leq \int_M \abs{\inner{\frac{\sqrt{\Sc_g}}{2} \psi, \frac{\sqrt{\Sc_g}}{2}\sigma_0}} &\leq \int_M \frac{\Sc_g}{4}\pnorm{\psi}^2 \int_M \frac{\Sc_g}{4} \pnorm{\sigma_0}^2.
    \end{align*}
    The integral
    \begin{equation*}
        \int_M \frac{\Sc_g}{4} \pnorm{\sigma_0}^2
    \end{equation*}
    is finite since \(\Sc_g \in L^1\) and because \(\pnorm{{\sigma_0}}^2 \leq C\) uniformly, for some constant \(C\).
    In Lemma \ref{lem:invertibility of dirac operator} we showed
    \begin{equation*}
        \int_M \frac{\Sc_g}{4}\pnorm{\psi}^2 \leq C\norm{\psi}_{H_{-\tau}^1}^2
    \end{equation*}
    for some \(C\) independent of \(\psi\).
    This shows
    \begin{equation*}
        \int_M \frac{\Sc_g}{4} \inner{\;\cdot\;, \sigma_0}
    \end{equation*}
    is a continuous functional on \(H_{-\tau}^1(M, SM)\).
    By checking on smooth compactly supported sections, it is not hard to see that
    \begin{equation*}
        \int_M \inner{D\psi, D\sigma_0} = \int_M \inner{\conn[] \psi, \conn[] \sigma_0} + \int_M \frac{\Sc_g}{4} \inner{\psi, \sigma_0}
    \end{equation*}
    for all \(\psi \in H_{-\tau}^1(M, SM)\).
    On the other hand, by directly applying the Lichnerowicz formula and Stokes's theorem,
    \begin{equation*}
        \int_{D(R)} \big< D\sigma_1, D\sigma_0 \big> = \int_{D(R)} \big< \conn[]\sigma_1, \conn[]\sigma_0 \big> + \int_{D(R)} \frac{\Sc_g}{4} \big< \sigma_1, \sigma_0 \big> + \int_{S(R)} \big< L^\partial \sigma_1, \sigma_0 \big>.
    \end{equation*}
    Thus,
    \begin{equation*}
        \lim_{R \to \infty} \int_{S(R)} \big< L^\partial \sigma_1, \sigma_0 \big> = 0.
    \end{equation*}

    Combining all of the above, we deduce
    \begin{equation*}
        0 = \int_M \pnorm{D\sigma}^2 = \int_M \pnorm{\conn[] \sigma}^2 + \int_M \frac{\Sc_g}{4} \pnorm{\sigma}^2 + \lim_{R \to \infty} \int_{S(R)} \big< L^\partial \sigma_0, \sigma_0 \big>.
    \end{equation*}
    By a computation of Bartnik,
    \begin{equation*}
        \lim_{R \to \infty} \int_{S(R)} \big< L^\partial \sigma_0, \sigma_0 \big> = -c(n)m(g)
    \end{equation*}
    for some normalization constant \(c(n) > 0\).
    It follows that \(m(g) \geq 0\).

    To prove the second part of the theorem, suppose that \(m(g) = 0\).
    The above procedure allows us to construct a parallel spinor on \(M\) given any spinor that is constant at infinity.
    In particular, we may construct enough linearly independent parallel sections to conclude that \(g\) is flat.
\end{proof}

\bibliographystyle{plain}
\bibliography{ref}
\end{document}